\documentclass[12pt]{article} 
\usepackage{amsfonts,amsmath,latexsym,amssymb,mathrsfs,amsthm}
\usepackage{slashbox}
\usepackage{comment}

\let\OLDthebibliography\thebibliography
\renewcommand\thebibliography[1]{
  \OLDthebibliography{#1}
  \setlength{\parskip}{0pt}
  \setlength{\itemsep}{0pt plus 0.3ex}
}

\def\numberlikeadb{\global\def\theequation{\thesection.\arabic{equation}}}
\numberlikeadb
\newtheorem{theorem}{Theorem}[section]
\newtheorem{lemma}[theorem]{Lemma}

\newtheorem{corollary}[theorem]{Corollary}

\newtheorem{proposition}[theorem]{Proposition}
\newtheorem{remark}[theorem]{Remark}

\usepackage{color}

\usepackage{lscape}
\usepackage{caption}
\usepackage{multirow}
\begin{document}

\title{Bounds for modified Lommel functions of the first kind and their ratios}
\author{Robert E. Gaunt\footnote{Department of Mathematics, The University of Manchester, Oxford Road, Manchester M13 9PL, UK}}

\date{\today} 
\maketitle

\vspace{-5mm}

\begin{abstract}The modified Lommel function $t_{\mu,\nu}(x)$ is an important special function, but to date there has been little progress on the problem of obtaining functional inequalities for $t_{\mu,\nu}(x)$.  In this paper, we advance the literature substantially by obtaining a simple two-sided inequality for the ratio $t_{\mu,\nu}(x)/t_{\mu-1,\nu-1}(x)$ in terms of the ratio $I_\nu(x)/I_{\nu-1}(x)$ of modified Bessel functions of the first kind, thereby allowing one to exploit the extensive literature on bounds for this ratio.  We apply this result to obtain two-sided inequalities for the condition numbers $xt_{\mu,\nu}'(x)/t_{\mu,\nu}(x)$, the ratio $t_{\mu,\nu}(x)/t_{\mu,\nu}(y)$ and the modified Lommel function $t_{\mu,\nu}(x)$ itself that are given in terms of $xI_\nu'(x)/I_\nu(x)$, $I_\nu(x)/I_\nu(y)$ and $I_\nu(x)$, respectively.  The bounds obtained in this paper are quite accurate and often tight in certain limits.  As an important special case we deduce bounds for modified Struve functions of the first kind and their ratios, some of which are new, whilst others extend the range of validity of some results given in the recent literature. 
\end{abstract}




\noindent{{{\bf{AMS 2010 Subject Classification:}}} Primary 33C20; 26D07. Secondary 33C10

\section{Introduction}\label{intro}

The ratios of modified Bessel functions $I_{\nu}(x)/I_{\nu-1}(x)$ and $K_{\nu-1}(x)/K_{\nu}(x)$ arise throughout the applied sciences; see \cite{segura} and references therein for examples of application areas.  These ratios are also important computational tools in the construction of numerical algorithms for computing modified Bessel functions; see, for example, Algorithms 12.6 and 12.7 of \cite{ast07}.  There now exists a substantial literature on lower and upper bounds for these ratios of modified Bessel functions \cite{amos74,b15,g32,kg13,ifantis,il07,im78, ln10,nasell2,pm50, rs16,segura,s12,yz17}.  There is also a extensive literature on bounds for the ratios $I_\nu(x)/I_\nu(y)$ and $K_\nu(x)/K_\nu(y)$, see, for example, \cite{amos74, b09, baricz2, ifantis, jbb, l91, paris, si95}, which has been used, for example, to obtain tight bounds for the generalized Marcum Q-function, which arises in signal processing \cite{b09}.  

The modified Struve and modified Lommel functions are related to the modified Bessel functions, and, likewise, they arise in manyfold applications.  A list of further application areas of modified Struve functions is given in \cite{bp13}, and, amongst other applications, Modified Lommel functions arise in the study of steady-state temperature distribution \cite{l75}, the vortex theory of screw propellers \cite{g29}, scattering amplitudes in quantum optics \cite{t73} and stress distributions in cylindrical objects \cite{s85}.

The first detailed study of inequalities for modified Struve functions was \cite{jn98}, in which two-sided inequalities for modified Struve functions and their ratios were obtained, together with Tur\'{a}n and Wronski type inequalities.  Recently, \cite{bp14,bps17,gaunt ineq5} have used results from the extensive study of modified Bessel functions and their ratios to obtain monotonicity results and functional inequalities for the modified Struve function of the first kind $\mathbf{L}_\nu(x)$ (some similar results for the modified Struve function of the second kind $\mathbf{M}_\nu(x)=\mathbf{L}_\nu(x)-I_\nu(x)$ are given in \cite{bp142}).  In particular, \cite{gaunt ineq5} obtained a simple but accurate two-sided inequality for the ratio $\mathbf{L}_{\nu}(x)/\mathbf{L}_{\nu-1}(x)$ in terms of the ratio $I_\nu(x)/I_{\nu-1}(x)$ of modified Bessel functions of the first kind.  This result allows one to take advantage of the literature on bounds for the ratio $I_\nu(x)/I_{\nu-1}(x)$ to obtain a variety of accurate bounds for $\mathbf{L}_\nu(x)/\mathbf{L}_{\nu-1}(x)$.  \cite{gaunt ineq5} also obtained two-sided inequalities for the ratio $\mathbf{L}_\nu(x)/\mathbf{L}_\nu(y)$ and the condition numbers $x\mathbf{L}_\nu'(x)/\mathbf{L}_\nu(x)$ in terms of the analogous terms involving $I_\nu(x)$.

The modified Lommel function $t_{\mu,\nu}(x)$ is an important special function (see Section \ref{seclom} for references and basic properties), which generalises the modified Struve function $\mathbf{L}_\nu(x)$ (see also Section \ref{seclom}) and has numerous applications, but despite this interest only \cite{css18} and the recent preprint \cite{m17} have touched on the problem of obtaining inequalities for this function.  In \cite{css18}, accurate bounds for the function $t_{\mu,\nu}(x)$ are obtained, but these bounds are only valid for $0<x<1$; whilst in \cite{m17}, some monotonicity properties of modified Lommel functions were established and a Redheffer type bound was obtained for the function $t_{\mu-1/2,1/2}(x)$. Some difficulties in obtaining inequalities for modified Lommel functions arise from the fact that, despite sharing close analogues of certain standard properties of $I_\nu(x)$ and $\mathbf{L}_\nu(x)$ (see again Section \ref{seclom}), in some circumstances the modified Lommel function $t_{\mu,\nu}(x)$ is more difficult to work with.  For example, the integral representation of $t_{\mu,\nu}(x)$ (see \cite{s36}) is more complicated than those of $I_\nu(x)$ and $\mathbf{L}_\nu(x)$ (see Sections 10.32(i) and 11.5(i) of \cite{olver}).  This means that the proof of the main result, Theorem 2.1, of \cite{gaunt ineq5} which bounds $\mathbf{L}_\nu(x)/\mathbf{L}_{\nu-1}(x)$ in terms of $I_\nu(x)/I_{\nu-1}(x)$ cannot easily be adapted to the ratio $t_{\mu,\nu}(x)/t_{\mu-1,\nu-1}(x)$.

In this paper, we obtain the first bounds in the literature  for the ratios of modified Lommel functions $t_{\mu,\nu}(x)/t_{\mu-1,\nu-1}(x)$ and $t_{\mu,\nu}(x)/t_{\mu,\nu}(y)$, the condition numbers $xt_{\mu,\nu}'(x)/t_{\mu,\nu}(x)$ and the function $t_{\mu,\nu}(x)$ itself (except for the special case considered by \cite{css18} and \cite{m17}).  The starting point for these inequalities is Theorem \ref{thmil}, in which we obtain a simple but accurate two-sided inequality for  $t_{\mu,\nu}(x)/t_{\mu-1,\nu-1}(x)$ in terms of $I_\nu(x)/I_{\nu-1}(x)$.  This result is quite powerful because it allows one to exploit the extensive literature on bounds for $I_\nu(x)/I_{\nu-1}(x)$ to bound $t_{\mu,\nu}(x)/t_{\mu-1,\nu-1}(x)$.  We thus progress the literature from having no bounds for the ratio $t_{\mu,\nu}(x)/t_{\mu-1,\nu-1}(x)$ to a wide variety, and we note some examples.  As a special case of Theorem \ref{thmil}, we obtain the same two-sided inequality for the ratio $\mathbf{L}_\nu(x)/\mathbf{L}_{\nu-1}(x)$ that was obtained in Theorem 2.1 of \cite{gaunt ineq5}, but with a larger range of validity for the lower bound.  This arises from an alternative method of proof that avoids the use of integral representations of the functions $\mathbf{L}_\nu(x)$ and $I_\nu(x)$ which were used to prove Theorem 2.1 of \cite{gaunt ineq5}.  In Section \ref{sec3}, we apply our bounds for $t_{\mu,\nu}(x)/t_{\mu-1,\nu-1}(x)$ to obtain a number of further functional inequalities involving $t_{\mu,\nu}(x)$.  Amongst other results, we obtain two-sided inequalities for $xt_{\mu,\nu}'(x)/t_{\mu,\nu}(x)$, $t_{\mu,\nu}(x)/t_{\mu,\nu}(y)$ and $t_{\mu,\nu}(x)$ that are given in terms of $xI_\nu'(x)/I_\nu(x)$, $I_\nu(x)/I_\nu(y)$ and $I_\nu(x)$, respectively.  
Through a combination of asymptotic analysis of the bounds and numerical results, we find that, in spite of their simple form, the bounds obtained in this paper are quite accurate and often tight in certain limits.

\section{Modified Lommel functions}\label{seclom}

In this section, we collect some basic properties of modified Lommel functions that will be needed in the sequel. Particular solutions of the modified Lommel differential equation \cite{s36,r64} (the Lommel differential equation was originally studied by \cite{l75})
\begin{equation}\label{3740a}x^2f''(x)+xf'(x)-(x^2+\nu^2)f(x)=x^{\mu+1}
\end{equation}
are the modified Lommel function of the first kind $t_{\mu,\nu}(x)$, as given by the hypergeometric series
\begin{align}t_{\mu,\nu}(x)&=\frac{x^{\mu+1}}{(\mu-\nu+1)(\mu+\nu+1)} {}_1F_2\bigg(1;\frac{\mu-\nu+3}{2},\frac{\mu+\nu+3}{2};\frac{x^2}{4}\bigg)\nonumber \\
\label{ldefn}&=2^{\mu-1}\Gamma\big(\tfrac{\mu-\nu+1}{2}\big)\Gamma\big(\tfrac{\mu+\nu+1}{2}\big)\sum_{k=0}^\infty\frac{(\frac{1}{2}x)^{\mu+2k+1}}{\Gamma\big(k+\frac{\mu-\nu+3}{2}\big)\Gamma\big(k+\frac{\mu+\nu+3}{2}\big)},
\end{align}
and the modified Lommel function of the second kind $T_{\mu,\nu}(x)$, as given by
\begin{equation}\label{Teqn}T_{\mu,\nu}(x)=t_{\mu,\nu}(x) -2^{\mu-1}\Gamma\big(\tfrac{\mu-\nu+1}{2}\big)\Gamma\big(\tfrac{\mu+\nu+1}{2}\big)I_\nu(x).
\end{equation}
In the literature, different notation is used for the modified Lommel functions; we adopt that used by \cite{zs13}. The terminology modified Lommel function of the \emph{first kind} and \emph{second kind} is not standard in the literature, but we consider it to be natural and adopt it for the following reasons.  Firstly, as we shall shortly see, the modified Lommel functions of the first and second kind generalise (up to  a multiplicative constant) the modified Struve functions of the first and second kind.  Secondly, \cite{bk16} have used the terminology Lommel function of the \emph{first kind} for the Lommel function $s_{\mu,\nu}(x)$.  Since $t_{\mu,\nu}(x)=-\mathrm{i}^{1-\mu}s_{\mu,\nu}(\mathrm{i}x)$ \cite{r64,zs13}, in analogy with the terminology for the Bessel, Struve and modified Bessel and Struve functions, it is natural to call $t_{\mu,\nu}(x)$ a modified Lommel function of the first kind, and $T_{\mu,\nu}(x)$ a modified Lommel function of the second kind. At this stage, we note that the focus of this paper is the function $t_{\mu,\nu}(x)$; we introduce $T_{\mu,\nu}(x)$ solely to fix notation and to allow us to state a result from the literature that will be required in the sequel. 

As the modified Lommel functions arise as Lommel functions with imaginary argument, many of their properties can be inferred from those of Lommel functions that are given in references, such as \cite{b67,cy11,l62,l69,olver,w44}.  The following relations involving modified Lommel functions that were given by \cite{r64} will be needed in the sequel:
\begin{align*}t_{\mu+2,\nu}(x)&=x^{\mu+1}+[(\mu+1)^2-\nu^2]t_{\mu,\nu}(x), \\
\frac{2\nu}{x}t_{\mu,\nu}(x)&=(\mu+\nu-1)t_{\mu-1,\nu-1}(x)-(\mu-\nu+1)t_{\mu-1,\nu+1}(x), \\
2t_{\mu,\nu}'(x)&= (\mu+\nu-1)t_{\mu-1,\nu-1}(x)+(\mu-\nu+1)t_{\mu-1,\nu+1}(x). 
\end{align*}
Combining these expressions gives the relations
\begin{align}\label{lomrel1}(\mu+\nu-1)t_{\mu-1,\nu-1}(x)-\frac{1}{\mu+\nu+1}t_{\mu+1,\nu+1}(x)=\frac{2\nu}{x}t_{\mu,\nu}(x)+\frac{x^\mu}{\mu+\nu+1}, \\
\label{lomrel2}(\mu+\nu-1)t_{\mu-1,\nu-1}(x)+\frac{1}{\mu+\nu+1}t_{\mu+1,\nu+1}(x)=2t_{\mu,\nu}'(x)-\frac{x^\mu}{\mu+\nu+1}.
\end{align}
The modified Lommel function $t_{\mu,\nu}(x)$ is related to the modified Bessel function $I_\nu(x)$ through the indefinite integral formula (see \cite{d59,r64} and apply (\ref{Teqn}))
\begin{equation}\label{lommelint}\int x^\mu I_\nu(x)\,\mathrm{d}x=x\big((\mu+\nu-1)I_\nu(x)t_{\mu-1,\nu-1}(x)-I_{\nu-1}(x)t_{\mu,\nu}(x)\big).
\end{equation}

For the purposes of this paper, the following normalization is particularly convenient and will remove a number of multiplicative constants from our calculations:
\begin{align}\label{tildet}\tilde{t}_{\mu,\nu}(x)&=\frac{1}{2^{\mu-1}\Gamma\big(\frac{\mu-\nu+1}{2}\big)\Gamma\big(\frac{\mu+\nu+1}{2}\big)}t_{\mu,\nu}(x) \\
\label{tseries}&=\sum_{k=0}^\infty\frac{(\frac{1}{2}x)^{\mu+2k+1}}{\Gamma\big(k+\frac{\mu-\nu+3}{2}\big)\Gamma\big(k+\frac{\mu+\nu+3}{2}\big)}. 
\end{align}
Analogously, we introduce the normalization
\begin{equation}\label{3740}\widetilde{T}_{\mu,\nu}(x)=\frac{1}{2^{\mu-1}\Gamma\big(\frac{\mu-\nu+1}{2}\big)\Gamma\big(\frac{\mu+\nu+1}{2}\big)}T_{\mu,\nu}(x)=\tilde{t}_{\mu,\nu}(x)-I_\nu(x).
\end{equation}
For ease of exposition, we shall also refer to $\tilde{t}_{\mu,\nu}(x)$ and $\widetilde{T}_{\mu,\nu}(x)$ as modified Lommel functions of the first and second kind.  From now on, we shall work with the functions $\tilde{t}_{\mu,\nu}(x)$ and $\widetilde{T}_{\mu,\nu}(x)$; it is easy to translate results back to $t_{\mu,\nu}(x)$ and $T_{\mu,\nu}(x)$.  For example, using (\ref{tildet}) and the standard formula $u\Gamma(u)=\Gamma(u+1)$ gives that
\begin{align*}\frac{t_{\mu,\nu}(x)}{t_{\mu-1,\nu-1}(x)}=(\mu+\nu-1)\frac{\tilde{t}_{\mu,\nu}(x)}{\tilde{t}_{\mu-1,\nu-1}(x)}.
\end{align*}

It is evident from (\ref{tseries}) that
\begin{equation*}\tilde{t}_{\mu,-\nu}(x)=\tilde{t}_{\mu,\nu}(x).
\end{equation*}
Also, from the power series representations of the modified Bessel function of the first kind $I_\nu(x)$ and the modified Struve function of first kind $\mathbf{L}_\nu(x)$, as given by
\begin{equation*}I_\nu(x)=\sum_{k=0}^\infty\frac{(\frac{1}{2}x)^{2k+\nu}}{k!\Gamma(k+\nu+1)}, \quad \mathbf{L}_\nu(x)=\sum_{k=0}^\infty\frac{(\frac{1}{2}x)^{2k+\nu+1}}{\Gamma(k+\frac{3}{2})\Gamma(k+\nu+\frac{3}{2})},
\end{equation*}
we can record the important special cases
\begin{equation*}\tilde{t}_{\nu-2n-1,\nu}(x)=I_\nu(x),\:  n=0,1,2,\ldots, \quad \tilde{t}_{\nu,\nu}(x)= \mathbf{L}_\nu(x), \quad \widetilde{T}_{\nu,\nu}(x)= \mathbf{M}_\nu(x),
\end{equation*}
where $\mathbf{M}_\nu(x)=\mathbf{L}_\nu(x)-I_\nu(x)$ is the modified Struve function of the second kind.  With our normalization, formulas (\ref{lomrel1}), (\ref{lomrel2}) and (\ref{lommelint}) become 
\begin{align}\label{struveid1}\tilde{t}_{\mu-1,\nu-1}(x)-\tilde{t}_{\mu+1,\nu+1}(x)&=\frac{2\nu}{x}\tilde{t}_{\mu,\nu}(x)+a_{\mu,\nu}(x), \\
\label{struveid2}\tilde{t}_{\mu-1,\nu-1}(x)+\tilde{t}_{\mu+1,\nu+1}(x)&=2\tilde{t}_{\mu,\nu}'(x)-a_{\mu,\nu}(x),
\end{align}
where $a_{\mu,\nu}(x)=\frac{(x/2)^\mu}{\Gamma(\frac{\mu-\nu+1}{2})\Gamma(\frac{\mu+\nu+3}{2})}$ (in the exceptional cases that $\mu-\nu=-1,-3,-5,\ldots$ or $\mu+\nu=-3,-5,-7,\ldots$ we have $a_{\mu,\nu}(x)=0$), and
\begin{equation}\label{lommelint2}\int x^\mu I_\nu(x)\,\mathrm{d}x=2^{\mu-1}\Gamma\big(\tfrac{\mu-\nu+1}{2}\big)\Gamma\big(\tfrac{\mu+\nu+1}{2}\big)x\big(I_\nu(x)\tilde{t}_{\mu-1,\nu-1}(x)-I_{\nu-1}(x)\tilde{t}_{\mu,\nu}(x)\big).
\end{equation}
The function $\tilde{t}_{\mu,\nu}(x)$ has the following asymptotic behaviour:
\begin{align}\label{ttend0}\tilde{t}_{\mu,\nu}(x)&\sim\frac{(\frac{1}{2}x)^{\mu+1}}{\Gamma\big(\frac{\mu-\nu+3}{2}\big)\Gamma\big(\frac{\mu+\nu+3}{2}\big)}\bigg(1+\frac{x^2}{(\mu+3)^2-\nu^2}\bigg), \quad x\downarrow0,\:\mu>-3,\:|\nu|<\mu+3,  \\
\label{linfty}\tilde{t}_{\mu,\nu}(x)&\sim\frac{\mathrm{e}^x}{\sqrt{2\pi x}}\bigg(1-\frac{4\nu^2-1}{8x}+\frac{(4\nu^2-1)(4\nu^2-9)}{128x^2}\bigg), \quad x\rightarrow\infty, \:\mu,\nu\in\mathbb{R}.
\end{align}
Here, (\ref{ttend0}) follows directly from (\ref{tseries}), and (\ref{linfty}) can be obtained by using the large $x$ asymptotic expansions of generalized hypergeometric functions \cite[Section 16.11]{olver}.

We end this section by recording that the modified Bessel function of the first kind satisfies the relations
\begin{align}\label{123e} I_{\nu-1}(x)- I_{\nu+1}(x)&=\frac{2\nu}{x} I_{\nu}(x), \\
\label{456e} I_{\nu-1}(x)+ I_{\nu+1}(x)&=2 I_{\nu}'(x),
\end{align}
and has the following asymptotic behaviour:
\begin{align}\label{itend0}I_\nu(x)&\sim \frac{x^\nu}{2^\nu\Gamma(\nu+1)}, \quad x\downarrow0,\:\nu\not=-1,-2,-3\ldots,\\
\label{iinfty}I_\nu(x)&\sim\frac{\mathrm{e}^x}{\sqrt{2\pi x}}\bigg(1-\frac{4\nu^2-1}{8x}+\frac{(4\nu^2-1)(4\nu^2-9)}{128x^2}\bigg), \quad x\rightarrow\infty,\: \nu\in\mathbb{R}
\end{align}
(see \cite{olver} for these and further properties).  It is instructive to note the similarity between properties (\ref{123e})--(\ref{iinfty}) of the modified Bessel function $I_\nu(x)$ and the analogous properties of the modified Lommel function $\tilde{t}_{\mu,\nu}(x)$.   Indeed, the asymptotic expansions (\ref{linfty}) and (\ref{iinfty}) are identical.  Thus, we see that  the modified Lommel function of the first kind $\tilde{t}_{\mu,\nu}(x)$ is the dominant solution to (\ref{3740a}) as $x\rightarrow\infty$ (with asymptotic expansion identical to that of the modified Bessel function of the first kind), and that the modified Lommel function of the second kind $\widetilde{T}_{\mu,\nu}(x)$ (as given by (\ref{3740})) must be recessive, as is the case for the modified Bessel function of the second kind $K_\nu(x)$.  This provides another sense in which the terminology modified Lommel functions of the ``first and second kind" seems to be natural.




\section{Upper and lower bounds for the ratio $t_{\mu,\nu}(x)/t_{\mu-1,\nu-1}(x)$}\label{sec2}

\subsection{Preliminaries and first results}\label{sec2.1}

We will need the following old result of \cite{bk55} in the sequel.  The importance of this result and its consequences in a general setting was first realised by \cite{pv97}.

\begin{lemma}\label{lem2.1}Suppose that the power series $f(x) =\sum_{n\geq0}a_nx^n$ and $g(x) =\sum_{n\geq0}b_nx^n$, where $a_n \in\mathbb{R}$ and $b_n > 0$ for all $n\geq0$, both converge on $(-r, r)$, $r > 0$. If the sequence $\{a_n/b_n\}_{n\geq0}$ is increasing (decreasing), then the function $x \mapsto f(x)/g(x)$ is also increasing (decreasing) on $(0, r)$.
\end{lemma}

We note that Lemma \ref{lem2.1} equivalently also holds true when both the power series $f(x)$ and $g(x)$ are even, or both are odd functions.

In the following proposition, we give a simple inequality for the modified Lommel function $\tilde{t}_{\mu,\nu}(x)$ that will be needed in the proof of Lemma \ref{blemma}, which in turn will be used in the proof of Theorem \ref{thmalm}. When $\mu=\nu$, the inequality reduces to an inequality of \cite{bp14} for the modified Struve function $\mathbf{L}_\nu(x)$, and our method of proof follows theirs.  We also note that similar results for the modified Lommel function of the first kind $t_{\mu,\nu}(x)$ are given in parts (iv) and (v) of Theorem 2.1 of \cite{m17}.


\begin{proposition}\label{propint}For $x>0$,
\begin{equation}\label{sinhineq}\tilde{t}_{\mu,\nu}(x)\leq \frac{x^\mu\sinh(x)}{2^{\mu+1}\Gamma\big(\frac{\mu-\nu+3}{2}\big)\Gamma\big(\frac{\mu+\nu+3}{2}\big)},
\end{equation}
if $\mu\geq-\frac{1}{2}$ and $(\mu+3)^2-\nu^2\geq6$.  The inequality is reversed if $-3<\mu\leq-\frac{1}{2}$, $(\mu+3)^2-\nu^2\leq6$ and $|\nu|<\mu+3$, and we have equality if and only if $\mu=\nu=-\frac{1}{2}$.
\end{proposition}

\begin{proof}From the series representation (\ref{tseries}) of $\tilde{t}_{\mu,\nu}(x)$ and $\sinh(x)=\sum_{k=0}^\infty\frac{x^{2k+1}}{(2k+1)!}$, we can write
\begin{equation*}\frac{x^\mu\sinh(x)}{\tilde{t}_{\mu,\nu}(x)}=2^{\mu+1}\frac{\sum_{k=0}^\infty\alpha_{k}x^{2k}}{\sum_{k=0}^\infty\beta_{\mu,\nu,k} x^{2k}},
\end{equation*}
where
\begin{equation*} \alpha_k=\frac{1}{(2k+1)!}\quad \text{and} \quad \beta_{\mu,\nu,k}=\frac{1}{2^{2k}\Gamma\big(k+\frac{\mu-\nu+3}{2}\big)\Gamma\big(k+\frac{\mu+\nu+3}{2}\big)}.
\end{equation*}
We now let $q_k=\alpha_k/\beta_{\mu,\nu,k}$ and calculate that
\begin{equation*}\frac{q_{k+1}}{q_k}=\frac{\big(k+\frac{\mu-\nu+3}{2}\big)\big(k+\frac{\mu+\nu+3}{2}\big)}{(k+\frac{3}{2})(k+1)}.
\end{equation*}
For a given $k=0,1,2,\ldots$, this ratio is $\geq1$ if
\[\mu^2+6\mu-\nu^2+3+2k(2\mu+1)\geq0,\]
and therefore $q_{k+1}/q_k\geq1$ for all $k=0,1,2,\ldots$ if $\mu\geq-\frac{1}{2}$ and $(\mu+3)^2-\nu^2\geq6$.  Similarly, $q_{k+1}/q_k\leq1$ for all $k=0,1,2,\ldots$ if $\mu\leq-\frac{1}{2}$ and $(\mu+3)^2-\nu^2\leq6$.  The assumptions on $\mu$ and $\nu$ in the statement of the proposition imply that $\mu-\nu>-3$ and $\mu+\nu>-3$, which ensures that all coefficients in the power series of $\tilde{t}_{\mu,\nu}(x)$ are positive.  Now, as the radius of convergence of the power series of $\sinh(x)$ and $\tilde{t}_{\mu,\nu}(x)$ is infinity, it follows from Lemma \ref{lem2.1} that the function $x\mapsto x^\mu\sinh(x)/\tilde{t}_{\mu,\nu}(x)$ is increasing on $(0,\infty)$ if $\mu\geq-\frac{1}{2}$ and $(\mu+3)^2-\nu^2\geq6$, and decreasing on $(0,\infty)$ if $-3<\mu\leq-\frac{1}{2}$, $(\mu+3)^2-\nu^2\leq6$ and $|\nu|<\mu+3$.  Thus, on computing the limit $\lim_{x\downarrow0}x^\mu\sinh(x)/\tilde{t}_{\mu,\nu}(x)$ using (\ref{ttend0}) we obtain inequality (\ref{sinhineq}) and its reverse.
\end{proof}


We now introduce a function that will appear throughout this paper. For $\mu>-2$ and $|\nu+1|<\mu+2$, we define the function $x\mapsto b_{\mu,\nu}(x):(0,\infty)\rightarrow(0,\frac{1}{2}(\mu-\nu+1))$ by
\begin{equation}\label{bdefn}b_{\mu,\nu}(x):=\frac{xa_{\mu,\nu}(x)}{2 \tilde{t}_{\mu,\nu}(x)}=\frac{(\frac{1}{2}x)^{\mu+1}}{\Gamma\big(\frac{\mu-\nu+1}{2}\big)\Gamma\big(\frac{\mu+\nu+3}{2}\big) \tilde{t}_{\mu,\nu}(x)}.
\end{equation}


\begin{lemma}\label{blemma}Let $\mu>-2$ and $|\nu+1|<\mu+2$ throughout this lemma. Then the following assertions are true:

\vspace{2mm}

\noindent (i) For fixed $\mu$ and $\nu$,
\begin{align}\label{bnu0}b_{\mu,\nu}(x)&\sim\frac{\mu-\nu+1}{2}\bigg(1-\frac{x^2}{(\mu+3)^2-\nu^2}\bigg), \quad x\downarrow0, \\
\label{bnu1}b_{\mu,\nu}(x)&\sim \frac{\sqrt{\pi}x^{\mu+3/2}\mathrm{e}^{-x}}{2^{\mu+1/2}\Gamma\big(\frac{\mu-\nu+1}{2}\big)\Gamma\big(\frac{\mu+\nu+3}{2}\big)}, \quad x\rightarrow\infty.
\end{align}

\vspace{2mm}

\noindent (ii) For fixed $\nu$ and $x>0$, $b_{\mu,\nu}(x)$ increases as $\mu$ increases in the interval $(\max\{\nu-1,-\nu-3\},\infty)$. 

\vspace{2mm}

\noindent (iii) For fixed $\mu$ and $\nu$,  $b_{\mu,\nu}(x)$ is a decreasing function of $x$ in $(0,\infty)$.  Therefore
\begin{equation*}b_{\mu,\nu}(x)<b_{\mu,\nu}(0^+)=\frac{\mu-\nu+1}{2}, \quad x>0.
\end{equation*}
This inequality can be improved further to
\begin{equation}\label{bcrude2}b_{\mu,\nu}(x)<\frac{\mu-\nu+1}{2}\bigg(1+\frac{x^2}{(\mu+3)^2-\nu^2}\bigg)^{-1}, \quad x>0.
\end{equation}

\noindent (iv) For $x>0$, 
\begin{equation}\frac{\mu-\nu+1}{2}x\mathrm{csch}(x)\label{star5}\leq b_{\mu,\nu}(x), 
\end{equation}
if $\mu\geq-\frac{1}{2}$ and $(\mu+3)^2-\nu^2\geq6$.  The inequality is reversed if $\mu\leq-\frac{1}{2}$ and $(\mu+3)^2-\nu^2\leq6$, and we have equality if and only if $\mu=\nu=-\frac{1}{2}$.
\end{lemma}


\begin{proof}(i) Combine (\ref{ttend0}) and (\ref{linfty}) with (\ref{bdefn}). 

\vspace{2mm}

\noindent (ii) A short calculation shows that
\begin{equation}\label{bdown}\frac{1}{b_{\mu,\nu}(x)}=\frac{2}{\mu-\nu+1}\sum_{k=0}^\infty\frac{\big(\frac{1}{2}x\big)^{2k}}{\big(\frac{\mu-\nu+3}{2}\big)_k\big(\frac{\mu+\nu+3}{2}\big)_k},
\end{equation}
where $(a)_k=a(a+1)\cdots(a+k-1)$ is the Pochhamer symbol. The assertion now follows.

\vspace{2mm} 

\noindent (iii) It is clear from (\ref{bdown}) that $b_{\mu,\nu}(x)$ is a decreasing function of $x$ in $(0,\infty)$ with $b_{\mu,\nu}(0^+)=\frac{\mu-\nu+1}{2}$. Inequality (\ref{bcrude2}) follows from truncating the series expansion of $\tilde{t}_{\mu,\nu}(x)$ at the second term.

\vspace{2mm}

\noindent (iv) This is immediate from Proposition \ref{propint}.
\end{proof}

\subsection{Bounding $t_{\mu,\nu}(x)/t_{\mu,-1,\nu-1}(x)$ via bounds for $I_\nu(x)/I_{\nu-1}(x)$}\label{sec2.2}

The following theorem gives a two-sided inequality for $\tilde{t}_{\mu,\nu}(x)/\tilde{t}_{\mu-1,\nu-1}(x)$ in terms of the ratio $I_\nu(x)/I_{\nu-1}(x)$, which constitutes the main result of this paper.  The importance of the result is not only due to the fact that one is are able to immediately exploit the literature on bounds for $I_\nu(x)/I_{\nu-1}(x)$ to bound $\tilde{t}_{\mu,\nu}(x)/\tilde{t}_{\mu-1,\nu-1}(x)$, but because, as we shall see in Section \ref{sec3}, the inequality will be used to obtain bounds for the condition numbers $x\tilde{t}_{\mu,\nu}'(x)/\tilde{t}_{\mu,\nu}(x)$, the ratio $\tilde{t}_{\mu,\nu}(x)/\tilde{t}_{\mu,\nu}(y)$ and the function $\tilde{t}_{\mu,\nu}(x)$ itself. 

The theorem generalises Theorem 2.2 of \cite{gaunt ineq5}, which gives analogous results for the modified Struve function $\mathbf{L}_\nu(x)$.  The method of proof of \cite{gaunt ineq5}, which is based on integral representations of $\mathbf{L}_\nu(x)$ and $I_\nu(x)$, does not easily generalise, owing to the fact that the corresponding integral representation of $\tilde{t}_{\mu,\nu}(x)$ is significantly more complicated than that of $\mathbf{L}_{,\nu}(x)$ (see \cite{s36}).  We are, however, able to provide an alternative simple proof, which has the advantage of extending the range of validity of the bounds of \cite{gaunt ineq5}; see Corollary \ref{corde} and Remark \ref{remde}.   




\begin{theorem}\label{thmil}(i) For $x>0$,
\begin{align}\label{bounddo}I_{\nu}(x) \tilde{t}_{\mu-1,\nu-1}(x)-I_{\nu-1}(x) \tilde{t}_{\mu,\nu}(x)>0,
\end{align}
which holds if $\mu>-1$, $\nu\geq-1$, $|\nu|<\mu+1$, or $\mu>-1$, $\nu>-1$, $-\mu-1\leq\nu<\mu+1$.  We have equality in (\ref{bounddo}) if $\mu-\nu=-1$, or $\nu=-1$, $\mu=0$ (for which $\mu+\nu=-1$).

Also, for $x>0$,
\begin{align}\label{bound456}I_{\nu}(x) \tilde{t}_{\mu-1,\nu-1}(x)-I_{\nu-1}(x) \tilde{t}_{\mu,\nu}(x)& < a_{\mu,\nu}(x)I_{\nu}(x),\\
\label{bound789}I_{\nu}(x) \tilde{t}_{\mu-1,\nu-1}(x)-I_{\nu-1}(x) \tilde{t}_{\mu,\nu}(x)&< a_{\mu-1,\nu-1}(x)I_{\nu-1}(x),
\end{align}
where $a_{\mu,\nu}(x)=\frac{(x/2)^\mu}{\Gamma(\frac{\mu-\nu+1}{2})\Gamma(\frac{\mu+\nu+3}{2})}$ ($a_{\mu,\nu}(x)=0$ if $\mu-\nu=-1$ or $\mu+\nu=-3$).  Inequality (\ref{bound456}) holds for $\mu>-2$, $\nu\geq-2$, $|\nu+1|<\mu+2$, or $\mu>-2$, $\nu>-2$, $-\mu-2\leq\nu+1<\mu+2$, and we have equality if $\mu-\nu=-1$, or $\mu=-1$, $\nu=-2$.  Inequality (\ref{bound789}) holds for $\mu>0$, $\nu\geq0$, $|\nu-1|<\mu$, or $\mu>0$, $\nu>0$, $-\mu\leq\nu-1<\mu$, with equality if $\mu-\nu=-1$, or $\mu=1$, $\nu=0$.




\vspace{2mm}

\noindent (ii) For $x>0$,
\begin{equation}\label{firstcor1}\bigg(\frac{I_{\nu-1}(x)}{I_\nu(x)}+\frac{2b_{\mu,\nu}(x)}{x}\bigg)^{-1}<\frac{\tilde{t}_{\mu,\nu}(x)}{\tilde{t}_{\mu-1,\nu-1}(x)}<\frac{I_{\nu}(x)}{I_{\nu-1}(x)}, 
\end{equation}
where both the lower and upper bounds hold for $\mu>-1$, $0\leq\nu<\mu+1$.  We have equality in the upper bound if $\mu-\nu=-1$, $\nu\geq0$.  
\end{theorem}


\begin{proof}(i) We focus on proving inequality (\ref{bounddo}); inequalities (\ref{bound456}) and (\ref{bound789}) follow from substituting the relations (\ref{struveid1}) and (\ref{123e}) into (\ref{bounddo}) and performing some basic manipulations.

Let us first prove inequality (\ref{bounddo}) under the assumption $\mu>-1$, $\nu\geq-1$, $|\nu|<\mu+1$.  For such $\mu$ and $\nu$, we can use the indefinite integral formula (\ref{lommelint2}) to obtain 
\begin{equation}\label{potd}\int_0^x u^\mu I_\nu(u)\,\mathrm{d}u=2^{\mu-1}\Gamma\big(\tfrac{\mu-\nu+1}{2}\big)\Gamma\big(\tfrac{\mu+\nu+1}{2}\big)x\big(I_\nu(x)\tilde{t}_{\mu-1,\nu-1}(x)-I_{\nu-1}(x)\tilde{t}_{\mu,\nu}(x)\big),
\end{equation}
where we used that $\lim_{x\downarrow0}xI_\nu(x)\tilde{t}_{\mu-1,\nu-1}(x)=\lim_{x\downarrow0}xI_{\nu-1}(x)\tilde{t}_{\mu,\nu}(x)=0$ (which can be seen from the series representations of $I_\nu(x)$ and $\tilde{t}_{\mu,\nu}(x)$).  Here, the assumptions $\mu+\nu>-1$ and $\nu\geq-1$ ensure that the integral exists and is in fact positive, since $I_{\nu}(x)>0$ for all $x>0$ if $\nu\geq-1$, and, as $x\downarrow0$, $I_\nu(x)=O(x^{\nu})$ for all $\nu>-1$ and $I_{-1}(x)=I_1(x)=O(x)$.  Furthermore, as $\mu-\nu>-1$ and $\mu+\nu>-1$, the gamma functions in (\ref{potd}) are positive, and we thus deduce inequality (\ref{bounddo}) from (\ref{potd}). 

We now prove that inequality (\ref{bounddo}) is also valid for $\mu>-1$, $\nu>-1$, $-\mu-1\leq\nu<\mu+1$.  Given what we have already shown, it suffices to consider the case $\nu>-1$, $-\mu-1=\nu<\mu+1$.  Note that $-1<\nu<0$.  In this case, as $\tilde{t}_{\mu,-\nu}(x)=\tilde{t}_{\mu,\nu}(x)$ and $\tilde{t}_{\nu-3,\nu}(x)=\tilde{t}_{\nu-1,\nu}(x)=I_\nu(x)$, we have
\begin{align}I_{\nu}(x) \tilde{t}_{\mu-1,\nu-1}(x)-I_{\nu-1}(x) \tilde{t}_{\mu,\nu}(x)&=I_{\nu}(x) \tilde{t}_{-\nu-2,\nu-1}(x)-I_{\nu-1}(x) \tilde{t}_{-\nu-1,\nu}(x) \nonumber \\
\label{onezero}&=I_\nu(x)I_{-(\nu-1)}(x)-I_{\nu-1}(x)I_{-\nu}(x)\\
&=-\frac{2\sin(\pi\nu)}{\pi x}>0. \nonumber
\end{align}
To obtain the final equality we used the standard formulas (see \cite{olver})
\begin{align*}I_{-\nu}(x)-I_\nu(x)=\frac{2\sin(\pi\nu)}{\pi}K_\nu(x), \quad I_\nu(x)K_{\nu+1}(x)+I_{\nu+1}(x)K_\nu(x)=\frac{1}{x},
\end{align*}
where $K_\nu(x)$ is the modified Bessel function of the second kind.

Finally, we consider the boundary cases in which there is equality in (\ref{bounddo}).  When $n$ is an integer, $I_{-n}(x)=I_n(x)$ and we can therefore see from (\ref{onezero}) that $I_{\nu}(x) \tilde{t}_{\mu-1,\nu-1}(x)-I_{\nu-1}(x) \tilde{t}_{\mu,\nu}(x)$ is equal to zero if $\mu=0$ and $\nu=-1$.  The final case to be considered is $\mu-\nu=-1$. Here, since $\tilde{t}_{\nu-1,\nu}(x)=I_\nu(x)$, we have 
\begin{align*}I_{\nu}(x) \tilde{t}_{\mu-1,\nu-1}(x)-I_{\nu-1}(x) \tilde{t}_{\mu,\nu}(x)&=I_{\nu}(x) \tilde{t}_{\nu-2,\nu-1}(x)-I_{\nu-1}(x) \tilde{t}_{\nu-1,\nu}(x) \\
&=I_\nu(x)I_{\nu-1}(x)-I_{\nu-1}(x)I_\nu(x)=0,
\end{align*}
which completes the proof of part (i).




\vspace{2mm}

\noindent (ii) The two-sided inequality follows from rearranging inequalities (\ref{bounddo}) and (\ref{bound456}) using the facts that $I_\nu(x)>0$ for all $x>0$ if $\nu\geq-1$, and $\tilde{t}_{\mu,\nu}(x)>0$ for all $x>0$ if $\mu-\nu\geq-3$ and $\mu+\nu\geq-3$ (this can be seen from the series representation of $\tilde{t}_{\mu,\nu}(x)$), together with the formula $b_{\mu,\nu}(x)=\frac{xa_{\mu,\nu}(x)}{2 \tilde{t}_{\mu,\nu}(x)}$.  The lower bound is valid for $\mu>-2$, $\nu\geq0$, $|\nu+1|<\mu+2$ and the upper bound is valid for $\mu>-1$, $\nu\geq0$, $|\nu|<\mu+1$, but both these constraints reduce to $\mu>-1$, $0\leq\nu<\mu+1$, as stated in the theorem.  
\end{proof}


\begin{corollary}\label{corde}Define $a_\nu(x):=a_{\nu,\nu}(x)=\frac{(x/2)^\nu}{\sqrt{\pi}\Gamma(\nu+\frac{3}{2})}$ (where $a_{-\frac{3}{2}}(x)=0$) and $b_\nu(x):=b_{\nu,\nu}(x)=\frac{xa_{\nu}(x)}{2 \mathbf{L}_{\nu}(x)}$.  Then:

\vspace{2mm}

\noindent (i) For $x>0$,
\begin{align}\label{qwert1}I_{\nu}(x)\mathbf{L}_{\nu-1}(x)-I_{\nu-1}(x)\mathbf{L}_{\nu}(x)&>0, \quad \nu\geq-\tfrac{1}{2}, \\
\label{qwert2}I_{\nu}(x)\mathbf{L}_{\nu-1}(x)-I_{\nu-1}(x)\mathbf{L}_{\nu}(x)& <a_\nu(x)I_{\nu}(x), \quad \nu\geq-\tfrac{3}{2},\\
\label{qwert3}I_{\nu}(x)\mathbf{L}_{\nu-1}(x)-I_{\nu-1}(x)\mathbf{L}_{\nu}(x)&<a_{\nu-1}(x)I_{\nu-1}(x), \quad \nu\geq\tfrac{1}{2}.
\end{align}

\vspace{2mm}

\noindent (ii) Let $\nu\geq0$.  Then, for $x>0$,
\begin{equation}\label{firstcor1z}\bigg(\frac{I_{\nu-1}(x)}{I_\nu(x)}+\frac{2b_\nu(x)}{x}\bigg)^{-1}<\frac{\mathbf{L}_{\nu}(x)}{\mathbf{L}_{\nu-1}(x)}<\frac{I_{\nu}(x)}{I_{\nu-1}(x)}.
\end{equation} 
\end{corollary}

\begin{proof}Set $\mu=\nu$ in Theorem \ref{thmil} and use that $\tilde{t}_{\nu,\nu}(x)=\mathbf{L}_\nu(x)$.   
\end{proof}

\begin{remark}\label{remde}Corollary \ref{corde} extends the range of validity of the bounds of Theorem 2.1 of \cite{gaunt ineq5}.  One can write down analogous corollaries to Theorems \ref{thmcond} and \ref{thmalm} below, which extend the range of validity of some of the bounds of Theorems 3.1 and 3.3 of \cite{gaunt ineq5} that bound the condition numbers $x\mathbf{L}_{\nu}'(x)/\mathbf{L}_{\nu}(x)$, and the ratio $\mathbf{L}_{\nu}(x)/\mathbf{L}_{\nu}(y)$ and modified Struve function $\mathbf{L}_\nu(x)$, respectively. For space reasons, we omit these corollaries, and just record that the range of validity of inequalities (3.34) and (3.43) of \cite{gaunt ineq5} are extended in this paper.     On the other hand, for $\mu=\nu$ the range of validity of inequalities (\ref{upper1}), (\ref{star10}), (\ref{thmu9}) and (\ref{near33}) are the same as those for the corresponding inequalities of \cite{gaunt ineq5}.


The range of validity for all inequalities of Corollary \ref{corde} is best possible.  To see this, we first recall that $\mathbf{M}_\nu(x)=\mathbf{L}_\nu(x)-I_\nu(x)$ has the asymptotic behaviour \cite[11.6.2]{olver}
\begin{equation*}\mathbf{M}_\nu(x)\sim-\frac{(\frac{1}{2}x)^{\nu-1}}{\sqrt{\pi}\Gamma(\nu+\frac{1}{2})}, \quad \nu\not=-\tfrac{1}{2},-\tfrac{3}{2},-\tfrac{5}{2},\ldots, \: x\rightarrow\infty.
\end{equation*}
Together with (\ref{iinfty}) we then have that, for $\nu\not=-\frac{1}{2},-\frac{3}{2},-\frac{5}{2}\ldots$,
\begin{align*}I_{\nu}(x)\mathbf{L}_{\nu-1}(x)-I_{\nu-1}(x)\mathbf{L}_{\nu}(x)&=I_{\nu}(x)\mathbf{M}_{\nu-1}(x)-I_{\nu-1}(x)\mathbf{M}_{\nu}(x) \\
&\sim\frac{\mathrm{e}^x}{\sqrt{2\pi x}}\cdot\frac{(\frac{1}{2}x)^{\nu-1}}{\sqrt{\pi}\Gamma(\nu+\frac{1}{2})},
\end{align*}
as $x\rightarrow\infty$.  Now, for $\nu=-\frac{1}{2}-\epsilon$, where $0<\epsilon<1$, we have $\Gamma(\nu+\frac{1}{2})<0$.  Thus, the range of validity of inequality (\ref{qwert1}) is best possible.  Similar considerations show that the range of validity of inequalities (\ref{qwert2}) and (\ref{qwert3}) are best possible.  Finally, the range of validity of (\ref{firstcor1z}) cannot be extended, because $I_{\nu}(x)$ takes negative values for $\nu\in(-2,-1)$.
\end{remark}

\begin{remark}Lower and upper bounds for integrals of the form $\int_0^xu^\mu I_\nu(u)\,\mathrm{d}u$ are given in \cite{gaunt ineq1, gaunt ineq3, gaunt ineq6}.  These bounds are given in terms of the modified Bessel function of the first kind and a number of both the lower and upper bounds are tight in the limits $x\downarrow0$ and $x\rightarrow\infty$.  It is therefore possible to obtain more accurate bounds for the quantities $I_{\nu}(x) \tilde{t}_{\mu-1,\nu-1}(x)-I_{\nu-1}(x) \tilde{t}_{\mu,\nu}(x)$ and $\tilde{t}_{\mu,\nu}(x)/\tilde{t}_{\mu,\nu}(x)$ than those given in Theorem \ref{thmil}.  However, the price one pays for this is that the resulting bounds are more complicated and consequently more difficult to work with than the two-sided inequality (\ref{firstcor1}), which is used throughout this paper.  

\end{remark}

\begin{remark}\label{remaa}Let $l_{\mu,\nu}^a(x)$, $u_{\mu,\nu}^a(x)$ denote the lower and upper bounds of (\ref{firstcor1}) and let $h_{\mu,\nu}(x)=\tilde{t}_{\mu,\nu}(x)/\tilde{t}_{\mu-1,\nu-1}(x)$.  From the asymptotic formulas (\ref{iinfty}) and (\ref{bnu1}), we have
\begin{equation*}\frac{u_{\mu,\nu}^a(x)}{l_{\mu,\nu}^a(x)}-1=\frac{2b_{\mu,\nu}(x)}{x}\frac{I_\nu(x)}{I_{\nu-1}(x)}=O(x^{\mu+1/2}\mathrm{e}^{-x}), \quad x\rightarrow\infty,
\end{equation*}
and so the double inequality (\ref{firstcor1}) is tight as $x\rightarrow\infty$.  From (\ref{itend0}) and (\ref{ttend0}), we have that in the limit $x\downarrow0$ the relative error in approximating $h_{\mu,\nu}(x)$ by $u_{\mu,\nu}^a(x)$ is $\frac{\mu-\nu+1}{2\nu}$. This error blows up as $\nu\downarrow0$, but the bound is tight when $\nu\rightarrow\infty$ and $\mu/\nu\rightarrow1$ simultaneously. The relative error is equal to 0 if $\mu-\nu=-1$, as we know must be so from the final assertion of Theorem \ref{thmil}. The lower bound $l_{\mu,\nu}^a(x)$ is tight as $x\downarrow0$.  To gain further insight,  we use (\ref{itend0}), (\ref{ttend0}) and (\ref{bnu0}) to obtain that, as $x\downarrow0$,
\begin{align}l_{\mu,\nu}^a(x)&\sim\frac{x}{\mu+\nu+1}-\frac{((\mu-\nu)^2+4\mu+7)x^3}{2(\nu+1)(\mu+\nu+1)^2((\mu+3)^2-\nu^2)}, \nonumber\\
\label{huvzero} h_{\mu,\nu}(x)&\sim\frac{x}{\mu+\nu+1}-\frac{2x^3}{(\mu+\nu+1)^2((\mu+3)^2-\nu^2)}.
\end{align}
The second term in the expansion of $l_{\mu,\nu}^a(x)$ approaches that of $h_{\mu,\nu}(x)$ when $\nu\rightarrow\infty$ and $(\mu-\nu)^2/\nu\rightarrow0$ simultaneously.  

Numerical results, obtained using \emph{Mathematica}, are given in Tables \ref{table1} and \ref{table11}.  We consider three cases of $\mu-\nu=k$, $k=-0.5,2,5$, and in each case let $\nu=0,1,2.5,5,10$.  (Tables for the case $\mu=\nu$ are given in \cite{gaunt ineq5}.)  In agreement with the above analysis, we see that, as a result of the exponential decay of $b_{\mu,\nu}(x)$, the bounds are very accurate for larger values of $x$, particularly for smaller values of $\mu$, whilst both bounds improve for `small' $x$ as $\nu$ increases.  The bounds are most accurate in the case $\mu-\nu=-0.5$.  
\end{remark}








\begin{table}[h]
\centering
\caption{\footnotesize{Relative error in approximating $ \tilde{t}_{\mu,\nu}(x)/ \tilde{t}_{\mu-1,\nu-1}(x)$ by the lower bound in (\ref{firstcor1}).}}
\label{table1}
{\scriptsize
\begin{tabular}{|c|rrrrrrrr|}
\hline
 \backslashbox{$(\mu,\nu)$}{$x$}      &    0.5 &    1 &    2.5 &    5 &    7.5 & 10 & 15 & 25  \\
 \hline
(-0.5,0)  & 3.8e$-$2 & 8.3e$-$2 & 5.8e$-$2 & 6.7e$-$3 & 5.8e$-$4 & 4.8e$-$5 & 3.3e$-$7 & 1.5e$-$11 \\
(0.5,1)  & 2.6e$-$3 & 8.8e$-$3 & 2.0e$-$2 & 7.3e$-$3 & 1.2e$-$3 & 1.5e$-$4 & 1.6e$-$6 & 1.3e$-10$ \\
(2,2.5)  & 4.2e$-$4 & 1.6e$-$3 & 5.9e$-$3 & 5.3e$-$3 & 1.7e$-$3 & 3.4e$-$4 & 7.5e$-$6 & 1.4e$-$9 \\
(4.5,5)  & 7.8e$-$5 & 3.0e$-$4 & 1.5e$-$3 & 2.5e$-$3 & 1.5e$-$3 & 5.7e$-$4 & 3.2e$-$5 & 2.0e$-$8 \\
(9.5,10) & 1.2e$-$5 & 4.8e$-$5 & 2.7e$-$4 & 6.9e$-$4   &  7.7e$-$4 & 5.3e$-$4 & 1.0e$-$4 & 4.3e$-$7 \\  
  \hline
(2,0)  & 2.3e$-$2 & 7.5e$-$2 & 1.8e$-$1 & 1.1e$-$1 & 3.1e$-$2 & 5.9e$-$3 & 1.2e$-$4 & 2.1e$-$8 \\
(3,1)  & 5.2e$-$3 & 1.9e$-$2 & 7.1e$-$2 & 7.7e$-$2 & 3.4e$-$2 & 9.1e$-$3 & 3.3e$-$4 & 9.4e$-$8 \\
(4.5,2.5)  & 1.3e$-$3 & 5.1e$-$3 & 2.4e$-$2 & 4.2e$-$2 & 2.9e$-$2 & 1.2e$-$2 & 7.6e$-$4 & 5.3e$-$7 \\
(7,5)  & 3.2e$-$4 & 1.2e$-$3 & 6.8e$-$3 & 1.7e$-$2 & 1.8e$-$2 & 1.2e$-$2 & 1.8e$-$3 & 4.1e$-$6 \\
(12,10) & 5.9e$-$5 & 2.3e$-$4 & 1.4e$-$3 & 4.3e$-$3 & 6.6e$-$3 & 6.7e$-$3 & 2.8e$-$3 & 4.4e$-$5 \\  
  \hline
(5,0)  & 1.5e$-$2 & 5.1e$-$2 & 1.6e$-$1 & 1.9e$-$1 & 1.3e$-$1 & 5.6e$-$2 & 4.2e$-$3 & 3.6e$-$6 \\
(6,1)  & 4.6e$-$3 & 1.7e$-$2 & 7.4e$-$2 & 1.2e$-$1 & 1.0e$-$1 & 5.8e$-$2 & 6.7e$-$3 & 1.0e$-$5 \\
(7.5,2.5)  & 1.5e$-$3 & 5.8e$-$3 & 3.0e$-$2 & 6.8e$-$2 & 7.3e$-$2 & 5.2e$-$2 & 1.0e$-$2 & 3.5e$-$5 \\
(10,5)  & 4.3e$-$4 & 1.7e$-$3 & 9.7e$-$3 & 2.8e$-$2 & 3.9e$-$2 & 3.7e$-$2 & 1.4e$-$2 & 1.5e$-$4 \\
(15,10) & 9.3e$-$5 & 3.7e$-$4 & 2.2e$-$3 & 7.6e$-$3 & 1.3e$-$2 & 1.6e$-$2 & 1.3e$-$2 & 7.7e$-$4 \\  
  \hline
\end{tabular}}
\end{table}

\begin{table}[h]
\centering
\caption{\footnotesize{Relative error in approximating $ \tilde{t}_{\mu,\nu}(x)/ \tilde{t}_{\mu-1,\nu-1}(x)$ by the upper bound in (\ref{firstcor1}).}}
\label{table11}
{\scriptsize
\begin{tabular}{|c|rrrrrrrr|}
\hline
 \backslashbox{$(\mu,\nu)$}{$x$}      &        0.5 &    1 &    2.5 &    5 &    7.5 & 10 & 15 & 25  \\
 \hline
(-0.5,0) & 3.8e$+$0 & 8.0e$-$1 & 4.7e$-$2 & 1.2e$-$3 & 5.1e$-$5 & 2.9e$-$6 & 1.2e$-$8 & 3.2e$-$13 \\
(0.5,1)  & 2.3e$-$1 & 1.9e$-$1 & 6.8e$-$2 & 6.8e$-$3 & 5.8e$-$4 & 4.8e$-$5 & 3.3e$-$7 & 1.5e$-$11 \\
(2,2.5)  & 9.8e$-$2 & 9.1e$-$2 & 5.6e$-$2 & 1.4e$-$2 & 2.2e$-$3 & 2.8e$-$4 & 3.6e$-$6 & 3.6e$-$10 \\
(4.5,5)  & 4.9e$-$2 & 4.8e$-$2 & 3.8e$-$2 & 1.7e$-$2 & 5.0e$-$3 & 1.1e$-$3 & 3.4e$-$5 & 1.1e$-$8 \\
(9.5,10)  & 2.5e$-$2 & 2.5$-$2 & 2.2e$-$2 & 1.5e$-$2 & 7.4e$-$2 & 3.0e$-$3 & 2.8e$-$4 & 5.6e$-$7 \\  
  \hline
(2,0)  & 2.4e$+$1 & 5.9e$+$0 & 8.3e$-$1 & 1.1e$-$1 & 1.6e$-$2 & 2.1e$-$3 & 2.6e$-$5 & 2.6e$-$9 \\
(3,1)  & 1.4e$+$0 & 1.3e$+$0 & 6.4e$-$1 & 1.6e$-$1 & 3.4e$-$2 & 5.9e$-$3 & 1.2e$-$4 & 2.1e$-$8 \\
(4.5,2.5)  & 5.9e$-$1 & 5.6e$-$1 & 4.2e$-$1 & 1.8e$-$1 & 5.8e$-$2 & 1.5e$-$2 & 5.3e$-$4 & 2.0e$-$7 \\
(7,5)  & 3.0e$-$1 & 2.9e$-$1 & 2.5e$-$1 & 1.6e$-$1 & 7.7e$-$2 & 3.0e$-$2 & 2.4e$-$3 & 2.8e$-$6 \\
(12,10)  & 1.5e$-$1 & 1.5e$-$1 & 1.4e$-$1 & 1.1e$-$1 & 7.6e$-$2  & 4.4e$-$2  & 9.1e$-$3 &  6.3e$-$5 \\  
  \hline
(5,0)  & 4.9e$+$1 & 1.3e$+$1 & 2.2e$+$0 & 5.4e$-$1 & 1.6e$-$1 & 4.2e$-$2 & 1.8e$-$3 & 9.1e$-$7 \\
(6,1)  & 2.9e$+$0 & 2.6e$+$0 & 1.5e$+$0 & 5.5e$-$1 & 2.0e$-$1 & 6.6e$-$2 & 4.2e$-$3 & 3.6e$-$6 \\
(7.5,2.5)  & 1.2e$+$0 & 1.1e$+$0 & 9.1e$-$1 & 5.0e$-$1 & 2.4e$-$1 & 9.8e$-$2 & 1.0e$-$2 & 1.8e$-$5 \\
(10,5)  & 6.0e$-$1 & 5.9e$-$1 & 5.3e$-$1 & 3.8e$-$1 & 2.4e$-$1 & 1.3e$-$1 & 2.4e$-$2 & 1.3e$-$4 \\
(15,10)  & 3.0e$-$1 & 3.0e$-$1 & 2.8e$-$1 & 2.4e$-$1 & 1.9e$-$1 & 1.3e$-$1 & 4.8e$-$2 & 1.3e$-$3 \\  
  \hline
\end{tabular}}
\end{table}

We now note two applications of the double inequality (\ref{firstcor1}) to deduce bounds for the ratio $\tilde{t}_{\mu,\nu}(x)/\tilde{t}_{\mu-1,\nu-1}(x)$ from existing bounds for $I_\nu(x)/I_{\nu-1}(x)$. It was shown in \cite{ifantis} that, for $x>0$,
\begin{equation}\label{tanhb}\frac{x\tanh(x)}{x+(2\nu-1)\tanh(x)}\leq\frac{I_\nu(x)}{I_{\nu-1}(x)}\leq\tanh(x),
\end{equation}
where both the lower and upper bounds are valid for $\nu\geq\frac{1}{2}$ and we have equality if and only if $\nu=\frac{1}{2}$, and by \cite{segura} (see also \cite{amos74,ln10,rs16}) that, for $x>0$,
\begin{equation}\label{sqrtbb}\frac{x}{\nu-\frac{1}{2}+\sqrt{\big(\nu+\frac{1}{2}\big)^2+x^2}}<\frac{I_{\nu}(x)}{I_{\nu-1}(x)}<\frac{x}{\nu-\frac{1}{2}+\sqrt{\big(\nu-\frac{1}{2}\big)^2+x^2}},
\end{equation}
where the lower bound holds for $\nu\geq0$ and the upper bound is valid for $\nu\geq\frac{1}{2}$.  Combining with (\ref{firstcor1}) then gives the following corollary, the bounds of which will in turn be used in the proofs of Theorems \ref{thmcond} and \ref{thmalm}.   


\begin{corollary}\label{bghj}For $\mu>-\frac{1}{2}$, $\frac{1}{2}\leq\nu<\mu+1$ and $x>0$,
\begin{equation}\label{uppp}\frac{x\tanh(x)}{x+(2\nu-1+2b_{\mu,\nu}(x))\tanh(x)}<\frac{\tilde{t}_{\mu,\nu}(x)}{\tilde{t}_{\mu-1,\nu-1}(x)}<\tanh(x)<1.
\end{equation}
Also, for $x>0$
\begin{equation}\label{upper1}\frac{x}{\nu-\frac{1}{2}+2b_{\mu,\nu}(x)+\sqrt{\big(\nu+\frac{1}{2}\big)^2+x^2}}<\frac{\tilde{t}_{\mu,\nu}(x)}{\tilde{t}_{\mu-1,\nu-1}(x)}<\frac{x}{\nu-\frac{1}{2}+\sqrt{\big(\nu-\frac{1}{2}\big)^2+x^2}},
\end{equation}
where the lower bound holds for $\mu>-1$, $0\leq\nu<\mu+1$ and the upper bound holds for $\mu>-\frac{1}{2}$, $\frac{1}{2}\leq\nu<\mu+1$.   
\end{corollary}



\begin{remark}As $b_{\mu,\nu}(x)<\frac{\mu-\nu+1}{2}$ for all $x>0$, we can obtain the simpler lower bounds
\begin{eqnarray}\label{hopdf}\frac{\tilde{t}_{\mu,\nu}(x)}{\tilde{t}_{\mu-1,\nu-1}(x)}&>&\frac{\tanh(x)}{\mu+\nu+1}, \\
\label{hopdf0}\frac{\tilde{t}_{\mu,\nu}(x)}{\tilde{t}_{\mu-1,\nu-1}(x)}&>&\frac{x}{\mu+\frac{1}{2}+\sqrt{\big(\nu+\frac{1}{2}\big)^2+x^2}},
\end{eqnarray}
which have the same range of validity as (\ref{uppp}) and (\ref{upper1}), respectively.  Similar simplifications can be made to all bounds given in this paper.
\end{remark}


\begin{remark}\label{rem3112} 

All bounds given in Corollary \ref{bghj} are tight as $x\rightarrow\infty$ and of the correct asymptotic order as $x\downarrow0$.  For a comparison, an asymptotic analysis shows that, for fixed $\mu$, when $\frac{1}{2}\leq\nu<1$ the double inequality (\ref{uppp}) performs better in the limit $x\downarrow0$, whilst the double inequality (\ref{upper1}) is better in this limit for $\nu\geq1$, and is most accurate as $x\rightarrow\infty$, expect for the single case $\nu=\frac{1}{2}$.  Moreover, we used \emph{Mathematica} to observe that, for fixed $\mu$, if $\nu\geq1$ then both bounds of (\ref{upper1}) outperform those of (\ref{uppp}).  Except for when $\nu$ is quite close to $\frac{1}{2}$, we find (\ref{upper1}) to be preferable, and as this estimate will be used throughout this paper it is useful to gain some insight into the quality of the approximation.  

Denote the lower and upper bounds of (\ref{upper1}) by $l_{\mu,\nu}^b(x)$ and $u_{\mu,\nu}^b(x)$, respectively, and write $h_{\mu,\nu}(x)=\tilde{t}_{\mu,\nu}(x)/\tilde{t}_{\mu-1,\nu-1}(x)$.  From the asymptotic formula (\ref{ttend0}), we see that in the limit $x\downarrow0$ the relative error in approximating $h_{\mu,\nu}(x)$ by $u_{\mu,\nu}^b(x)$ is $\frac{\mu-\nu+2}{2\nu-1}$. The error blows up as $\nu\downarrow\frac{1}{2}$, but decreases as $\nu$ increases, and the bound is tight when $\nu\rightarrow\infty$ and $\mu/\nu\rightarrow1$ simultaneously.  The relative error in approximating $h_{\mu,\nu}(x)$ by $l_{\mu,\nu}^b(x)$ is 0 in the limit $x\downarrow0$, and furthermore
\begin{align*}l_{\mu,\nu}^b(x)\sim\frac{x}{\mu+\nu+1}-\frac{((\mu-\nu)^2+5\mu-\nu+8)x^3}{(2\nu+1)(\mu+\nu+1)^2((\mu+3)^2-\nu^2)}, \quad x\downarrow0. 
\end{align*}
From (\ref{huvzero}), we see that second term in the expansion of $l_{\mu,\nu}^b(x)$ approaches that of $h_{\mu,\nu}(x)$ when $\nu\rightarrow\infty$ and $(\mu-\nu)^2/\nu\rightarrow0$ simultaneously, as we observed for the lower bound of (\ref{firstcor1}). Also, as $x\rightarrow\infty$, $u_{\mu,\nu}^b(x)/l_{\mu,\nu}^b(x)-1\sim \nu/x^2$,
and so for `large' $x$ the accuracy of the double inequality (\ref{upper1}) decreases as $\nu$ increases.  The $O(x^{-2})$ error in the approximation is much larger than the $O(x^{\mu+1/2}\mathrm{e}^{-x})$ error of (\ref{firstcor1}).  These comments are supported by numerical results given in Tables \ref{table3} and \ref{table4}.
\end{remark}

\begin{table}[h]
\centering
\caption{\footnotesize{Relative error in approximating $ \tilde{t}_{\mu,\nu}(x)/ \tilde{t}_{\mu-1,\nu-1}(x)$ by the lower bound of (\ref{upper1}).}}
\label{table3}
{\scriptsize
\begin{tabular}{|c|rrrrrrrrr|}
\hline
 \backslashbox{$(\mu,\nu)$}{$x$}      &    0.5 &    1 &    2.5 &    5 &    7.5 & 10 & 15 & 25 & 50  \\
 \hline
(-0.5,0) & 1.6e$-$1 & 2.3e$-$1 & 1.1e$-$1 & 1.9e$-$2 & 5.7e$-$3 & 2.8e$-$3 & 1.2e$-$3 & 4.2e$-$4 & 1.0e$-$4 \\
(0.5,1) & 1.0e$-$2 & 3.1e$-$2 & 5.9e$-$2 &  2.9e$-$2 & 1.2e$-$2 & 6.7e$-$3 & 3.1e$-$3 & 1.1e$-$3 & 2.9e$-$4 \\
(2,2.5)  & 1.5e$-$3 & 5.4e$-$3 & 2.0e$-$2 & 2.4e$-$2 & 1.4e$-$2 & 9.3e$-$3 & 4.9e$-$3 & 2.0e$-$3 & 5.5e$-$4 \\
(4.5,5)  & 2.6e$-$4 & 1.0e$-$3 & 5.0e$-$3 & 1.1e$-$2 & 1.1e$-$2 & 8.8e$-$3 & 5.8e$-$3 & 2.9e$-$3 & 9.1e$-$4 \\ 
(9.5,10) & 3.9e$-$5 & 1.5e$-$4 & 8.8e$-$4 & 2.7e$-$3 & 4.1e$-$3 & 4.7e$-$3 & 4.6e$-$3 & 3.3e$-$3 & 1.4e$-$3 \\  
  \hline
(2,0) & 5.0e$-$2 & 1.2e$-$1 & 1.9e$-$1 & 1.2e$-$1 & 3.6e$-$2 & 8.6e$-$3 & 1.3e$-$3  & 4.2e$-$4 & 1.0e$-$4 \\
(3,1) & 9.0e$-$3 & 3.1e$-$2 & 9.4e$-$2 & 9.3e$-$2 & 4.4e$-$2 & 1.6e$-$2 & 3.4e$-$3 & 1.1e$-$3 & 2.9e$-$4 \\
(4.5,2.5)  & 2.0e$-$3 & 7.7e$-$3 & 3.4e$-$2 & 5.6e$-$2 & 4.1e$-$2 & 2.1e$-$2 & 5.6e$-$3 & 2.0e$-$3 & 5.5e$-$4 \\
(7,5)  & 4.6e$-$4 & 1.8e$-$3 & 9.7e$-$3 & 2.4e$-$2 & 2.6e$-$2 & 2.0e$-$2 & 7.5e$-$3 & 2.9e$-$3 & 9.1e$-$4 \\ 
(12,10) & 8.3e$-$5 & 3.3e$-$4 & 1.9e$-$3 & 6.2e$-$3 & 9.7e$-$3 & 1.1e$-$2 & 7.3e$-$3 & 3.4e$-$3 & 1.4e$-$3 \\  
  \hline
(5,0) & 2.8e$-$2 & 7.6e$-$2 & 1.7e$-$1 & 1.9e$-$1 & 1.3e$-$1 & 5.9e$-$2 & 5.4e$-$3 & 4.2e$-$4 & 1.0e$-$4 \\
(6,1) & 7.0e$-$3 & 2.5e$-$2 & 8.9e$-$2 & 1.3e$-$1 & 1.1e$-$1 & 6.4e$-$2 & 9.7e$-$3 & 1.2e$-$3 & 2.9e$-$4 \\
(7.5,2.5)  & 2.0e$-$3 & 7.7e$-$3 & 3.7e$-$2 & 7.8e$-$2 & 8.2e$-$2 & 6.0e$-$2  & 1.5e$-$2 & 2.0e$-$3 & 5.5e$-$4 \\
(10,5)  & 5.5e$-$4 & 2.2e$-$3 & 1.2e$-$2 & 3.4e$-$2 & 4.6e$-$2 & 4.4e$-$2 & 2.0e$-$2 & 3.1e$-$3 & 9.1e$-$4 \\ 
(15,10) & 1.1e$-$4 & 4.5e$-$4 & 2.7e$-$3 & 9.2e$-$3 & 1.6e$-$2 & 2.0e$-$2 & 1.7e$-$2 & 4.1$-$3 & 1.4e$-$3 \\  
  \hline
\end{tabular}}
\end{table}

\begin{table}[h]
\centering
\caption{\footnotesize{Relative error in approximating $ \tilde{t}_{\mu,\nu}(x)/ \tilde{t}_{\mu-1,\nu-1}(x)$ by the upper bound of (\ref{upper1}).}}
\label{table4}
{\scriptsize
\begin{tabular}{|c|rrrrrrrrr|}
\hline
 \backslashbox{$(\mu,\nu)$}{$x$}       &    0.5 &    1 &    2.5 &    5 &    7.5 & 10 & 15 & 25 & 50  \\
 \hline
(0.5,1)  & 1.1e$+$0 & 6.5e$-$1 & 1.5e$-$1 & 2.0e$-$2 & 5.8e$-$3 & 2.8e$-$3 & 1.2e$-$3 & 4.2e$-$4 & 1.0e$-$4 \\
(2,2.5)  & 3.6e$-$1 & 3.2e$-$1 & 1.8e$-$1 & 6.0e$-$2 & 2.1e$-$2 & 1.1e$-$2 & 4.7e$-$3 & 1.7e$-$3 & 4.1e$-4$ \\
(4.5,5)  & 1.6e$-$1 & 1.6e$-$1 & 1.3e$-$1 & 7.2e$-$2 & 3.8e$-$2 & 2.2e$-$2 & 9.8e$-$2 & 3.6e$-$3 & 9.1e$-$4 \\
(9.5,10) & 7.9e$-$2 & 7.8e$-$2 & 7.2e$-$2 & 5.7e$-$2 & 4.1e$-$2 & 2.9e$-$2 & 1.6e$-$2 & 6.9e$-$3 & 1.9e$-$3 \\  
  \hline
(3,1)  & 3.2e$+$0 & 2.1e$+$0 & 7.6e$-$1 & 1.8e$-$1 & 3.9e$-$2 & 8.7e$-$3 & 1.3e$-$3 & 4.2e$-$4 & 1.0e$-$4 \\
(4.5,2.5)  & 9.7e$-$1 & 9.0e$-$1 & 5.9e$-$1 & 2.3e$-$1 & 7.8e$-$2 & 2.6e$-$2 & 5.2e$-$3 & 1.7e$-$3 & 4.1e$-$4 \\
(7,5)  & 4.4e$-$1 & 4.3e$-$1 & 3.7e$-$1 & 2.2e$-$1 & 1.1e$-$1 & 5.1e$-$2 & 1.2e$-$2 & 3.6e$-$3 & 9.1e$-$4 \\
(12,10) & 2.1e$-$1 & 2.1e$-$1 & 2.0e$-$1 & 1.6e$-$1 & 1.1e$-$1 & 7.1e$-$2 & 2.5e$-$2 & 6.9e$-$3 & 1.9e$-$3 \\  
  \hline
(6,1)  & 5.6e$+$0 & 4.0e$+$0 & 1.7e$+$0 & 5.7e$-$1 & 2.1e$-$1 & 6.9e$-$2 & 5.4e$-$3 & 4.2e$-$4 & 1.0e$-$4 \\
(7.5,2.5)  & 1.7e$+$0 & 1.6e$+$0 & 1.1e$+$0 & 5.6e$-$1 & 2.6e$-$1 & 1.1e$-$1 & 1.5e$-$2 & 1.7e$-$3  & 4.1e$-4$ \\
(10,5)  & 7.7e$-$1 & 7.6e$-$1 & 6.7e$-$1 & 4.6e$-$1 & 2.8e$-$1 & 1.5e$-$1 & 3.4e$-$2 & 3.6e$-$3 & 9.1e$-$4 \\
(15,10) & 3.7e$-$1 & 3.7e$-$1  & 3.5e$-$1 & 3.0e$-$1 & 2.3e$-$1 & 1.6e$-$1 & 6.5e$-$2 & 8.2e$-$3 & 1.9e$-$3 \\  
  \hline
\end{tabular}}
\end{table}

\section{Further bounds for modified Lommel functions of the first kind and their ratios}\label{sec3}

In this section, we apply the bounds of Section \ref{sec2} for the ratio $ \tilde{t}_{\mu,\nu}(x)/ \tilde{t}_{\mu-1,\nu-1}(x)$ to obtain further functional inequalities for the modified Lommel function $\tilde{t}_{\mu,\nu}(x)$. 

\subsection{Bounds for the condition numbers}\label{sec3.2}

Following the notation of \cite{segura}, we write $C\big( \tilde{t}_{\mu,\nu}(x)\big)=x \tilde{t}'_{\mu,\nu}(x)/ \tilde{t}_{\mu,\nu}(x)$ and $C\big(I_\nu(x)\big)=xI'_\nu(x)/I_\nu(x)$.  These are positive quantities if $\mu-\nu\geq-3$ and $\mu+\nu\geq-3$, and $\nu\geq-1$, respectively. From (\ref{struveid1}) and (\ref{struveid2}) we obtain the relations
\begin{align}\label{star60}C\big( \tilde{t}_{\mu,\nu}(x)\big)&=\frac{x \tilde{t}_{\mu-1,\nu-1}(x)}{ \tilde{t}_{\mu,\nu}(x)}-\nu, \\
\label{star61}C\big( \tilde{t}_{\mu,\nu}(x)\big)&=\frac{x \tilde{t}_{\mu+1,\nu+1}(x)}{ \tilde{t}_{\mu,\nu}(x)}+\nu+2b_{\mu,\nu}(x),
\end{align}
and thus bounds for $\tilde{t}_{\mu,\nu}(x)/\tilde{t}_{\mu-1,\nu-1}(x)$ immediately lead to bounds for the condition number $C\big( \tilde{t}_{\mu,\nu}(x)\big)$.  In the following theorem, we give a two-sided inequality for $C\big( \tilde{t}_{\mu,\nu}(x)\big)$ in terms of $C\big(I_\nu(x)\big)$.  This result parallels inequality (\ref{firstcor1}) of Theorem \ref{thmil} by allowing one to use the literature on bounds for $C\big(I_\nu(x)\big)$ (see \cite{amos74,b092,g32,ln10,p99,pm50,segura}) to bound $C\big( \tilde{t}_{\mu,\nu}(x)\big)$.  The first inequalities for $C\big(I_\nu(x)\big)$, due to \cite{g32}, were motivated by a problem in wave mechanics, and some comments on the utility of the condition numbers $C\big(f(x)\big)=|xf'(x)/f(x)|$ for comparing functions are given in \cite{segura}.   





\begin{theorem}\label{thmcond}The following inequalities hold:

\vspace{2mm}

\noindent (i) For $x>0$,
\begin{equation}\label{condlo}C\big(I_\nu(x)\big)<C\big( \tilde{t}_{\mu,\nu}(x)\big)<C\big(I_\nu(x)\big)+2b_{\mu,\nu}(x),
\end{equation}
where the lower bound is valid for $\mu>-1$, $0\leq\nu<\mu+1$ and the upper bound is valid for $\mu>-2$, $-1\leq\nu<\mu+1$.

\vspace{2mm}

\noindent (ii) For $x>0$,
\begin{align}\label{star10}\sqrt{\big(\nu-\tfrac{1}{2}\big)^2+x^2}-\tfrac{1}{2}<C\big( \tilde{t}_{\mu,\nu}(x)\big)<\sqrt{\big(\nu+\tfrac{1}{2}\big)^2+x^2}+2b_{\mu,\nu}(x)-\tfrac{1}{2},
\end{align}
and
\begin{equation}\label{near44}x\coth(x)-\nu<C\big( \tilde{t}_{\mu,\nu}(x)\big)<x\tanh(x)+v+2b_{\mu,\nu}(x),
\end{equation}
where the lower bounds of (\ref{star10}) and (\ref{near44}) hold for $\mu>-\frac{1}{2}$, $\frac{1}{2}\leq\nu<\mu+1$, and the upper bounds of (\ref{star10}) and (\ref{near44}) hold for $\mu>-\frac{3}{2}$, $-\frac{1}{2}\leq\nu<\mu+1$.
\end{theorem}

\begin{proof}\noindent (i) From the relations (\ref{star60}) and (\ref{star61}) for $ \tilde{t}_{\mu,\nu}(x)$, as well as the corresponding relations (\ref{123e}) and (\ref{456e}) for $I_\nu(x)$, and the upper bound of inequality (\ref{firstcor1}), we obtain, for $\mu>-1$, $0\leq\nu<\mu+1$,
\begin{align*}\frac{x \tilde{t}_{\mu,\nu}'(x)}{ \tilde{t}_{\mu,\nu}(x)}=\frac{x \tilde{t}_{\mu-1,\nu-1}(x)}{ \tilde{t}_{\mu,\nu}(x)}-\nu>\frac{xI_{\nu-1}(x)}{I_\nu(x)}-\nu=\frac{xI_\nu'(x)}{I_\nu(x)},
\end{align*}
and, for $\mu>-2$, $-1\leq\nu<\mu+1$,
\begin{align*}\frac{x \tilde{t}_{\mu,\nu}'(x)}{ \tilde{t}_{\mu,\nu}(x)}=\frac{x \tilde{t}_{\mu+1,\nu+1}(x)}{ \tilde{t}_{\mu,\nu}(x)}+\nu+2b_{\mu,\nu}(x)<\frac{xI_{\nu+1}(x)}{I_\nu(x)}+\nu+2b_{\mu,\nu}(x)=\frac{xI_\nu'(x)}{I_\nu(x)}+2b_{\mu,\nu}(x).
\end{align*}

\vspace{2mm}

\noindent (ii) Combine the upper bounds of (\ref{uppp}) and (\ref{upper1}) with formulas (\ref{star60}) and (\ref{star61}).    
\end{proof}

\subsection{Bounds for the ratio $t_{\mu,\nu}(x)/t_{\mu,\nu}(y)$ and the modified Lommel function $t_{\mu,\nu}(x)$}\label{sec3.3}



In the spirit of Section \ref{sec3.2}, we note that some basic manipulations allow one to exploit the bounds of Section \ref{sec2} to bound the ratio $\tilde{t}_{\mu,\nu}(x)/\tilde{t}_{\mu,\nu}(y)$.  To this end, on integrating the relations
\begin{align*}\frac{ \tilde{t}_{\mu,\nu}'(u)}{ \tilde{t}_{\mu,\nu}(u)}=\frac{ \tilde{t}_{\mu-1,\nu-1}(u)}{ \tilde{t}_{\mu,\nu}(u)}-\frac{\nu}{u}, \quad
\frac{ \tilde{t}_{\mu,\nu}'(u)}{ \tilde{t}_{\mu,\nu}(u)}=\frac{ \tilde{t}_{\mu+1,\nu+1}(u)}{ \tilde{t}_{\mu,\nu}(u)}+\frac{\nu}{u}+\frac{a_{\mu,\nu}(u)}{ \tilde{t}_{\mu,\nu}(u)}
\end{align*}
between $x$ and $y$ one obtains
\begin{align}\label{star30}\frac{ \tilde{t}_{\mu,\nu}(x)}{ \tilde{t}_{\mu,\nu}(y)}&=\bigg(\frac{y}{x}\bigg)^\nu\exp\bigg(-\int_x^y\frac{ \tilde{t}_{\mu-1,\nu-1}(u)}{ \tilde{t}_{\mu,\nu}(u)}\,\mathrm{d}u\bigg), \\
\label{star305}\frac{ \tilde{t}_{\mu,\nu}(x)}{ \tilde{t}_{\mu,\nu}(y)}&=\bigg(\frac{x}{y}\bigg)^\nu\exp\bigg(-2\int_x^y\frac{b_{\mu,\nu}(u)}{u}\,\mathrm{d}u\bigg)\exp\bigg(-\int_x^y\frac{ \tilde{t}_{\mu+1,\nu+1}(u)}{ \tilde{t}_{\mu,\nu}(u)}\,\mathrm{d}u\bigg).
\end{align}
If $\mu>-2$ and $|\nu+1|<\mu+2$, then by inequality (\ref{bcrude2}) we have
\begin{align}\label{mncz}2\int_x^y\frac{b_{\mu,\nu}(u)}{u}\,\mathrm{d}u&<\int_x^y\frac{\mu-\nu+1}{u\big(1+\frac{1}{(\mu+3)^2-\nu^2}u^2\big)}\,\mathrm{d}u=(\mu-\nu+1)\log\Bigg(\!\frac{y}{x}\sqrt{\frac{(\mu+3)^2-\nu^2+x^2}{(\mu+3)^2-\nu^2+y^2}}\Bigg),
\end{align}
and substituting into (\ref{star305}) then yields the lower bound
\begin{equation}\label{star318}\frac{ \tilde{t}_{\mu,\nu}(x)}{ \tilde{t}_{\mu,\nu}(y)}>\bigg(\frac{x}{y}\bigg)^\nu\bigg(\frac{(\mu+3)^2-\nu^2+y^2}{(\mu+3)^2-\nu^2+x^2}\bigg)^{\frac{\mu-\nu+1}{2}}\exp\bigg(-\int_x^y\frac{ \tilde{t}_{\mu+1,\nu+1}(u)}{ \tilde{t}_{\mu,\nu}(u)}\,\mathrm{d}u\bigg).
\end{equation}
Also, if $\mu\geq-\frac{1}{2}$ and $(\mu+3)^2-\nu^2\geq6$, then using inequality (\ref{star5}) gives 
\begin{equation*}2\int_x^y\frac{b_{\mu,\nu}(u)}{u}\,\mathrm{d}u\geq(\mu-\nu+1)\int_x^y\mathrm{csch}(u)\,\mathrm{d}u=(\mu-\nu+1)\log\Bigg(\frac{\tanh\big(\frac{1}{2}x\big)}{\tanh\big(\frac{1}{2}y\big)}\Bigg),
\end{equation*}
and combining with (\ref{star305}) then yields the upper bound
\begin{equation}\label{star31}\frac{ \tilde{t}_{\mu,\nu}(x)}{ \tilde{t}_{\mu,\nu}(y)}\leq\bigg(\frac{x}{y}\bigg)^\nu\bigg(\frac{\tanh\big(\frac{1}{2}x\big)}{\tanh\big(\frac{1}{2}y\big)}\bigg)^{\mu-\nu+1}\exp\bigg(-\int_x^y\frac{ \tilde{t}_{\mu+1,\nu+1}(u)}{ \tilde{t}_{\mu,\nu}(u)}\,\mathrm{d}u\bigg).
\end{equation}
We combine (\ref{star30}), (\ref{star318}) and (\ref{star31}) with the results of Section \ref{sec2} to prove the following theorem.  In this theorem, we give several different bounds for $\tilde{t}_{\mu,\nu}(x)/\tilde{t}_{\mu,\nu}(y)$ and $\tilde{t}_{\mu,\nu}(x)$, some of which are based on the estimates of the doubles inequalities (\ref{tanhb}) and (\ref{sqrtbb}) of \cite{ifantis} and \cite{segura} for $I_{\nu}(x)/I_{\nu-1}(x)$.  However, we note that the extensive literature on bounds for $I_\nu(x)/I_\nu(y)$ (with a number of such bounds surveyed in \cite{baricz2}) allow for many other bounds to be readily obtained, which we omit for space reasons. 

\begin{theorem}\label{thmalm}The following inequalities hold:

\vspace{2mm}

\noindent (i) For $0<x<y$,
\begin{equation}\label{thmalm1}\bigg(\frac{x}{y}\bigg)^{\mu-\nu+1}\bigg(\frac{(\mu+3)^2-\nu^2+y^2}{(\mu+3)^2-\nu^2+x^2}\bigg)^{\frac{\mu-\nu+1}{2}}\frac{I_\nu(x)}{I_\nu(y)}<\frac{ \tilde{t}_{\mu,\nu}(x)}{ \tilde{t}_{\mu,\nu}(y)}<\frac{I_\nu(x)}{I_\nu(y)},
\end{equation}
where the lower bound holds for $\mu>-2$, $-1\leq\nu<\mu+1$ and the upper bound holds for $\mu>-1$, $0\leq\nu<\mu+1$.

\vspace{2mm}

\noindent (ii) For $0<x<y$,
\begin{align}&\frac{\mathrm{e}^{\sqrt{(\nu+1/2)^2+x^2}}}{\mathrm{e}^{\sqrt{(\nu+1/2)^2+y^2}}}\bigg(\frac{x}{y}\bigg)^{\mu+1}\bigg(\frac{(\mu+3)^2-\nu^2+y^2}{(\mu+3)^2-\nu^2+x^2}\bigg)^{\frac{\mu-\nu+1}{2}}\Bigg(\frac{\nu+\frac{1}{2}+\sqrt{(\nu+\frac{1}{2})^2+y^2}}{\nu+\frac{1}{2}+\sqrt{(\nu+\frac{1}{2})^2+x^2}}\Bigg)^{\nu+\frac{1}{2}}<\nonumber\\
\label{thmu9}&<\frac{ \tilde{t}_{\mu,\nu}(x)}{ \tilde{t}_{\mu,\nu}(y)}<\frac{\mathrm{e}^{\sqrt{(\nu+3/2)^2+x^2}}}{\mathrm{e}^{\sqrt{(\nu+3/2)^2+y^2}}}\bigg(\frac{\tanh\big(\frac{1}{2}x\big)}{\tanh\big(\frac{1}{2}y\big)}\bigg)^{\mu-\nu+1}\bigg(\frac{x}{y}\bigg)^{\nu}\Bigg(\frac{\mu+\frac{3}{2}+\sqrt{(\nu+\frac{3}{2})^2+y^2}}{\mu+\frac{3}{2}+\sqrt{(\mu+\frac{3}{2})^2+x^2}}\Bigg)^{\mu+\frac{3}{2}},
\end{align}
where the lower bound holds for $\mu>-\frac{3}{2}$, $-\frac{1}{2}\leq\nu<\mu+1$ and the upper bound holds for $\mu\geq-\frac{1}{2},$ $-1\leq\nu<\mu+1$, $(\mu+3)^2-\nu^2\geq6$.

\vspace{2mm}

\noindent (iii) For $x>0$,
\begin{align}\label{star271}I_\nu(x)<\bigg(\frac{x^2}{(\mu+3)^2-\nu^2+x^2}\bigg)^{-\frac{\mu-\nu+1}{2}}\tilde{t}_{\mu,\nu}(x)
<C_{\mu,\nu}I_\nu(x),
\end{align}
where
\[C_{\mu,\nu}=\frac{((\mu+3)^2-\nu^2)^{\frac{\mu-\nu+1}{2}}\Gamma(\nu+1)}{2^{\mu-\nu+1}\Gamma\big(\frac{\mu-\nu+3}{2}\big)\Gamma(\frac{\mu+\nu+3}{2}\big)}.\]
Here both the lower and upper bounds hold for $\mu>-2$, $-1<\nu<\mu+1$.  Also, 
\begin{equation}\label{star272}\tilde{t}_{\mu,\nu}(x)<I_\nu(x), \quad x>0,\: \mu>-1,\: 0\leq\nu<\mu+1.
\end{equation}

\vspace{2mm}

\noindent (iv) Let $\mu>-\frac{3}{2}$, $-\frac{1}{2}\leq\nu<\mu+1$. Then, for $x>0$,
\begin{align}\label{near33}\frac{1}{\sqrt{2\pi}}<\frac{\big((\mu+3)^2-\nu^2+x^2\big)^{\frac{\mu-\nu+1}{2}}\big(\nu+\frac{1}{2}+\sqrt{\big(\nu+\frac{1}{2}\big)^2+x^2}\big)^{\nu+\frac{1}{2}}}{x^{\mu+1}\mathrm{e}^{\sqrt{(\nu+1/2)^2+x^2}}}\tilde{t}_{\mu,\nu}(x)<C_{\mu,\nu}',
\end{align}
where 
\[C_{\mu,\nu}'=\frac{((\mu+3)^2-\nu^2)^{\frac{\mu-\nu+1}{2}}(\mathrm{e}^{-1}(2\nu+1))^{\nu+\frac{1}{2}}}{2^{\mu+1}\Gamma\big(\frac{\mu-\nu+3}{2}\big)\Gamma\big(\frac{\mu+\nu+3}{2}\big)}.\]
\end{theorem}


\begin{proof}(i) Let us first note the following integral formula (see \cite[p$.$ 561]{gaunt ineq5}):
\begin{align*}\int_x^y\frac{I_{\nu\pm1}(u)}{I_{\nu}(u)}\,\mathrm{d}u&=\log\bigg(\frac{I_{\nu}(y)}{I_{\nu}(x)}\bigg)\mp\nu\log\bigg(\frac{y}{x}\bigg).
\end{align*}
With the aid of this formula, combining the upper bound of (\ref{firstcor1}) with (\ref{star30}) and (\ref{star318}) leads to the upper and lower bounds, respectively.



\vspace{2mm}

\noindent (ii) We proceed as in part (i) and first note the integral formula
\begin{equation}\label{mjyu}\int_x^y\frac{u}{a+\sqrt{b^2+u^2}}\,\mathrm{d}u=\sqrt{b^2+y^2}-\sqrt{b^2+x^2} +a\log\bigg(\frac{a+\sqrt{b^2+x^2}}{a+\sqrt{b^2+y^2}}\bigg).
\end{equation}
With this formula at hand, applying inequality (\ref{hopdf0}) to (\ref{star31}) yields the upper bound, whilst combining the upper bound of (\ref{upper1}) with (\ref{star318}) yields the lower bound.

\vspace{2mm}


\noindent (iii) For the lower bound in (\ref{star271}) and inequality (\ref{star272}), let $y\rightarrow\infty$ in the double inequality (\ref{thmalm1}).  For the upper bound in (\ref{star271}), let $x\downarrow0$ in the lower bound of (\ref{thmalm1}), and then replace $y$ by $x$.  In computing the limits, we use (\ref{linfty}), (\ref{iinfty}) and (\ref{ttend0}).

\vspace{2mm}

\noindent (iv) Proceed as in part (iii) by taking appropriate limits in (\ref{thmu9}). 
\end{proof}

\begin{remark}\label{rem12}For fixed $y>0$, both bounds of (\ref{thmu9}) are $O(x^{\mu+1})$, as $x\downarrow0$.  For fixed $x>0$, as $y\rightarrow\infty$, the lower bound is $O(y^{1/2}\mathrm{e}^{-y})$, which is the correct order (see (\ref{linfty})), whereas the upper bound is $O(y^{3/2}\mathrm{e}^{-y})$.  That the upper bound is not of the correct order as $y\rightarrow\infty$ can be at least partly traced back to the use of the inequality $b_{\mu,\nu}(x)<\frac{\mu-\nu+1}{2}$ (which is very crude for large $x$) in obtaining inequality (\ref{hopdf0}) from the lower bound of (\ref{upper1}). In deriving the lower bound (\ref{thmu9}), we used the refined inequality (\ref{bcrude2}) to obtain inequality (\ref{mncz}), which enabled us to obtain the correct order as $y\rightarrow\infty$, but using this inequality to bound the lower bound of (\ref{upper1}) leads to an integral that is less tractable than (\ref{mjyu}).  
\end{remark}

\begin{remark}Given the extensive literature on bounds for ratios of modified Bessel functions of the first kind, many other inequalities can be obtained than those given in Theorem \ref{thmalm}.  For example, by applying inequalities (2.1) and (2.3) of \cite{ifantis} to the upper bound of (\ref{thmalm1}) and the lower bound of (\ref{star271}), respectively, we obtain
\begin{eqnarray*}\frac{\tilde{t}_{\mu,\nu}(x)}{\tilde{t}_{\mu,\nu}(y)}&<&\bigg(\frac{x}{y}\bigg)^\nu\bigg(\frac{\cosh(x)}{\cosh(y)}\bigg)^\frac{1}{2(\nu+1)}, \quad 0<x<y,\:\mu>-1,\:\tfrac{1}{2}\leq\nu<\mu+1, \\
\tilde{t}_{\mu,\nu}(x)&>&\frac{1}{2^\nu\Gamma(\nu+1)}\frac{x^{\mu+1}\cosh^{\frac{1}{2(\nu+1)}}(x)}{\big((\mu+3)^2-\nu^2+x^2\big)^{\frac{\mu-\nu+1}{2}}}, \quad x>0,\:\mu>-2,\: -\tfrac{1}{2}\leq\nu<\mu+1.
\end{eqnarray*}
\end{remark}

\begin{remark}Setting $\mu=\nu$ in (\ref{star271}) yields the upper bound
\begin{equation}\label{near22}\mathbf{L}_\nu(x)<\frac{\Gamma(\nu+1)\sqrt{3(2\nu+3)}}{\sqrt{\pi}\Gamma(\nu+\frac{3}{2})}\frac{xI_\nu(x)}{\sqrt{x^2+3(2\nu+3)}}, \quad x>0,\:\nu>-1,
\end{equation}
which complements the following inequality of \cite{bp14}:
\begin{equation}\label{bpstu}\mathbf{L}_\nu(x)\leq\frac{2\Gamma(\nu+2)}{\sqrt{\pi}\Gamma(\nu+\frac{3}{2})}I_{\nu+1}(x), \quad x>0, \:\nu\geq-\tfrac{1}{2},
\end{equation}
with equality if and only if $\nu=-\frac{1}{2}$. Both inequalities are tight as $x\downarrow0$ and are of the correct asymptotic order as $x\rightarrow\infty$.  An application of Stirling's inequality \cite[5.6.1]{olver} shows that the multiplicative constant in (\ref{near22}) is bounded between $\sqrt{6/\pi}$ and $\sqrt{6}$ for all $\nu\geq-\frac{1}{2}$, but that the constant in (\ref{bpstu}) is $O(\sqrt{\nu})$ as $\nu\rightarrow\infty$.
\end{remark}
 
\begin{remark}The upper bound in (\ref{near33}) generalises the upper bound of inequality (3.45) of \cite{gaunt ineq5}, which gives a bound for the modified Struve function $\mathbf{L}_\nu(x)$; see Remark 3.6 of that paper for comments on the performance of the bound.  Our upper bound is tight as $x\downarrow0$ and of the correct asymptotic order as $x\rightarrow\infty$, with relative error $\sqrt{2\pi}C'_{\mu,\nu}-1$ in this limit (recall that $\tilde{t}_{\mu,\nu}(x)\sim \frac{1}{\sqrt{2\pi x}}\mathrm{e}^x$ as $x\rightarrow\infty$).  The constant $C_{\mu,\nu}'$ is quite complicated, but we can gain some insight by using Stirling's approximation \cite[5.6.1]{olver}.  For fixed $\nu\geq-\frac{1}{2}$, $C_{\mu,\nu}'=O(\mu^{-\nu-1}\mathrm{e}^\mu)$ as $\mu\rightarrow\infty$.  Now, let $k=\mu-\nu$ be fixed.  Then, as $\mu\rightarrow\infty$,
\begin{equation*}C_{\mu,\nu}'\sim g(k):=\frac{(k+3)^{\frac{k+1}{2}}\big(\frac{\mathrm{e}}{2}\big)^{\frac{k}{2}}}{\sqrt{2\pi}\Gamma\big(\frac{k+3}{2}\big)}.
\end{equation*}  
One can check that $g$ is an increasing function of $k$ on $[-\frac{1}{2},\infty)$, with minimum value $g(-\frac{1}{2})=0.5125\ldots$.  This is consistent with our findings in Remarks \ref{remaa} and \ref{rem3112}, which suggest that the double inequalities (\ref{firstcor1}) and (\ref{upper1}) become more accurate as $k=\mu-\nu$ decreases.  

In contrast to the upper bound of (\ref{near33}), the lower bound is tight as $x\rightarrow\infty$ but only correct up to asymptotic order as $x\downarrow0$, with a relative error of $1-(\sqrt{2\pi}C_{\mu,\nu}')^{-1}$ in this limit.  On setting $\mu=\nu$ in the lower bound of (\ref{near33}) we obtain the new bound
\begin{equation}\label{llowerr}\mathbf{L}_\nu(x)>\frac{(2\pi)^{-\frac{1}{2}}x^{\mu+1}\mathrm{e}^{\sqrt{(\nu+1/2)^2+x^2}}}{\sqrt{3(2\nu+3)+x^2}\big(\nu+\frac{1}{2}+\sqrt{(\nu+\frac{1}{2})^2+x^2}\big)^{\nu+\frac{1}{2}}},\quad x>0, \:\nu\geq-\tfrac{1}{2},
\end{equation}
which when combined with the upper bound of inequality (3.45) of \cite{gaunt ineq5} gives a two-sided inequality for $\mathbf{L}_\nu(x)$ that is of the correct asymptotic order in both the limits $x\downarrow0$ and $x\rightarrow\infty$.  Numerical results (see Table \ref{table5}) suggest that, for fixed $\nu$, the relative error in approximating $\mathbf{L}_\nu(x)$ decreases from the initial value of $1-(\sqrt{2\pi}C_{\nu,\nu}')^{-1}$ at $x=0$ down to $0$ as $x$ increases. 

Lastly, we note that setting $x\downarrow0$ in the upper bound of (\ref{thmu9}) and then replacing $y$ by $x$ gives the following alternative to the lower bound of (\ref{near33}): 
\begin{equation*}\tilde{t}_{\mu,\nu}(x)>\frac{\mathrm{e}^{\sqrt{(\nu+3/2)^2+x^2}-\nu-3/2}}{2^{\nu}\Gamma\big(\frac{\mu-\nu+3}{2}\big)\Gamma\big(\frac{\mu+\nu+3}{2}\big)}x^\nu\tanh^{\mu-\nu+1}\Big(\frac{x}{2}\Big)\Bigg(\frac{2\mu+3}{\mu+\frac{3}{2}+\sqrt{\big(\nu+\frac{3}{2}\big)^2+x^2}}\Bigg)^{\mu+\frac{3}{2}},
\end{equation*}
$x>0$, $\mu\geq-\frac{1}{2},$ $-1\leq\nu<\mu+1$, $(\mu+3)^2-\nu^2\geq6$.  This inequality is tight as $x\downarrow0$, but is $O(x^{-3/2}\mathrm{e}^x)$ as $x\rightarrow\infty$, which is smaller than the $O(x^{-1/2}\mathrm{e}^x)$ rate of $\tilde{t}_{\mu,\nu}(x)$.



\begin{table}[h]
\centering
\caption{\footnotesize{Relative error in approximating $\mathbf{L}_{\nu}(x)$ by (\ref{llowerr}).}}
\label{table5}
{\scriptsize
\begin{tabular}{|c|rrrrrrrrrr|}
\hline
 \backslashbox{$\nu$}{$x$}      &   0.5 &    1 &    2.5 &    5 &     10 & 15 & 25 & 50 & 100 & 200 \\
 \hline
0 & 6.3e$-$1 & 5.7e$-1$ & 3.8e$-1$ & 1.8e$-1$  & 6.7e$-$2 & 3.6e$-2$ & 1.7e$-2$ & 6.8e$-1$3 & 3.0e$-3$ & 1.4e$-3$ \\
1 & 5.7e$-1$ & 5.6e$-$1 & 4.6e$-$1  & 2.9e$-1$ & 1.3e$-$1 & 7.8e$-$2 & 4.1e$-$2 & 1.8e$-2$ & 8.2e$-$3 & 3.9e$-$3 \\
2.5 & 5.4e$-$1 & 5.4e$-$1  & 4.9e$-$1  & 3.8e$-$1 & 2.1e$-$1 & 1.3e$-$1 & 7.3e$-$2 & 3.4e$-$2 & 1.6e$-$2 & 7.7e$-3$ \\
5 & 5.2e$-$1 & 5.2e$-$1  & 5.0e$-$1  & 4.5e$-$1 & 3.0e$-$1 & 2.1e$-$1 & 1.2e$-$1 & 5.8e$-$2 & 2.8e$-$2 & 1.4e$-$2 \\
10 & 5.1e$-$1 & 5.1e$-$1  & 5.0e$-$1  & 4.8e$-$1 & 3.9e$-$1 & 3.0e$-$1 & 2.0e$-$1 & 1.0e$-$1 & 5.2e$-$2 & 2.6e$-$2 \\  
\hline
\end{tabular}}
\end{table}

\end{remark}

\subsection*{Acknowledgements}
The author is supported by a Dame Kathleen Ollerenshaw Research Fellowship.  The author would like to thank the referee for their constructive comments and suggestions.

\end{document}